\documentclass{article}
 \usepackage{authblk}
\usepackage{a4wide, amstext, amsmath, amssymb, graphicx, array, epsfig, psfrag}
\usepackage{amstext, amsmath, amssymb, graphicx, todonotes}
\usepackage{stmaryrd}
\usepackage{color}
\usepackage[update,prepend]{epstopdf}

\newcommand{\bn}{\boldsymbol n}

\newcommand{\bg}{\boldsymbol g}

\newcommand{\bv}{\boldsymbol v}
\newcommand{\bnu}{\boldsymbol\nu}
\newcommand{\bu}{\boldsymbol u}
\newcommand{\bw}{\boldsymbol w}
\newcommand{\be}{\boldsymbol e}
\newcommand{\bft}{\boldsymbol t}

\newcommand{\ff}{\boldsymbol f}

\newcommand{\bfzeta}{\boldsymbol \zeta}
\newcommand{\srf}{\bfzeta}

\newcommand{\bsigma}{\boldsymbol{\sigma}}

\numberwithin{equation}{section}
\newtheorem{lem}{Lemma}[section]
\newtheorem{prop}{Proposition}[section]

\newtheorem{rem}{Remark}[section]

\newenvironment{proof}{\noindent \newline {\bf Proof.}}
{\hfill \mbox{\fbox{} } \newline}

\date{}
\begin{document}
\title{Application of a minimal compatible element to incompressible and nearly incompressible continuum mechanics}
\author[$\star$]{Erik Burman}
\author[$\dagger$]{Snorre Christiansen}
\author[$\ddagger$]{Peter Hansbo}
\affil[$\star$]{\footnotesize\it  Department of Mathematics, University College London, London, UK--WC1E  6BT, United Kingdom}
\affil[$\dagger$]{\footnotesize\it  Department of Mathematics,
University of Oslo,
P.O. Box 1053 Blindern,
NO-0316 Oslo, Norway}
\affil[$\ddagger$]{\footnotesize\it Department of Mechanical Engineering, J\"onk\"oping University, SE-55111 J\"onk\"oping, Sweden}
\maketitle
\begin{abstract}
In this note we will explore some applications of the recently
constructed piecewise affine, $H^1$-conforming element that fits in a discrete de Rham
complex ({\emph{Christiansen and Hu, Generalized finite element systems for smooth differential forms and Stokes' problem. Numer. Math. 140 (2018)}}). In particular we show how the element leads to locking free
methods for incompressible elasticity and viscosity robust methods for
the Brinkman model.
\end{abstract}
\section{Introduction}
It is well known that standard finite element methods are not in general
well-suited for the approximation of nearly incompressible elasticity
or incompressible flow problems. Indeed, in particular low order
approximation spaces often suffer from locking in the incompressible
limit \cite{BF91}. They may typically also exhibit instability when Darcy
flow is considered if the element was designed for Stokes' problem \cite{MaTaWi02}.
These problems can be alleviated using stabilization \cite{BH07,BH05}, but such
stabilizing terms, although weakly consistent to the right order, may
upset local conservation of e.g. mass, momentum, and introduce an additional
layer of complexity to the computational method and its analysis. Recently some new
results on $H^1$-conforming piecewise polynomial approximation spaces
compatible with the de
Rham complex have been published \cite{GN14a,GN14b,JLMNR17,GS18,FGN18, CH18}. Such elements are interesting, since
they provide a simple tool for the robust approximation of models in
mechanics where a divergence constraint is present. Herein we will
focus on the piecewise affine element derived in the last
reference. The advantage of this approach is that it offers a simple low
order locking free element in arbitrary space dimensions. Observe
that for the Scott-Vogelius element the polynomial order of the spaces
typically depends on the number of dimensions \cite{FGN18}. We discuss how this element can be implemented in
engineering practice and show the basic, robust, error estimates that
may be obtained for linear elasticity and incompressible flow. In this
paper we will consider two different models, linear elasticity and the
Brinkman model for porous media flow. The idea is to show the locking
free property of the element on the elasticity model and then
illustrate how the element seamlessly can change between the Stokes'
equations modelling free flow and Darcy's equations modelling porous
media flow, while remaining $H^1$-conforming. The two models are
introduced in section \ref{sec:models}. The construction of the
element is discussed in section \ref{sec:FEMconst} and the finite
element discretizations of the model problems and their analysis are
the topics of section \ref{sec:disc} and \ref{sec:estimate}. In section \ref{sec:nitsche}
we discuss how boundary conditions may be imposed weakly using
Nitsche's method, without sacrifying the good properties of the
element. Finally section \ref{sec:numerics} gives some numerical illustrations to
the theory.
\section{Model problems: linear elasticity and the Brinkman model}\label{sec:models}
We will consider two model problems with solutions in $V := [H^1(\Omega)]^d$, initially assuming
homogeneous Dirichlet boundary conditions.
Let $\Omega \subset \mathbb{R}^d$, $d=2,3$ denote a convex polyhedral domain with
boundary $\partial \Omega$. The first model problem is linear
elasticity. Here we wish to find $\bu \in V^0$, where $V^0:=V \cap [H^1_0(\Omega)]^d$, such that 
\begin{equation}\label{eq:elasticity}
-\nabla \cdot \bsigma(\bu) = \ff , \mbox{ in } \Omega,
\end{equation}
where $\bsigma(\bu) = 2 \mu \nabla^s \bu + \lambda \mathbb{I} \nabla
\cdot \bu$, with $\nabla^s$ the symmetric part of the gradient tensor,
$\mathbb{I}$ the identity matrix 
and $\mu,\, \lambda >0$ the Lam\'e coefficients and $\ff \in [L^2(\Omega)]^d$. This system can be written on weak form: find $\bu \in
V^0$
such that
\[
a_E(\bu,\bv) = l(\bv),\mbox{ for all } \bv \in V^0,
\]
where
\begin{equation}\label{eq:bilin_elasticity}
a_E(\bw,\bv):= \int_\Omega \bsigma(\bu) : \nabla^s \bv ~\mbox{d}x,
\end{equation}
where the tensor product is defined by $A:B := \sum_{i,j=1}^d a_{ij}
b_{ij}$ and 
\begin{equation}\label{eq:rhs}
l(\bv) := \int_\Omega \ff \cdot \bv ~\mbox{d}x.
\end{equation}
It is well-known that the problem \eqref{eq:elasticity} admits a
unique weak solution in the space $V^0$ through
application of Lax-Milgram's lemma, and that the following regularity
holds \cite{Gris11},
\begin{equation}\label{eq:elasticity_regularity}
\|\bu\|_{H^2(\Omega)} + \lambda \|\nabla \cdot
\bu\|_{H^1(\Omega)} \leq C_R \|\ff\|_\Omega \mbox{ for } \mu \in
[\mu_1,\mu_2] \mbox{ and } \lambda \in (0,\infty).
\end{equation}

The second model problem is the Brinkman problem where we look for a
velocity-pressure couple $(\bu,p) \in V^0 \times
Q$, where $Q:=L^2_0(\Omega)$ denotes the set of square
integrable functions with mean zero, such that
\begin{equation}\label{eq:brinkman}
\begin{array}{rcl}
-\mu \Delta \bu + \sigma \bu + \nabla p &=& \ff \mbox{ in } \Omega \\
\nabla \cdot \bu & = & g  \mbox{ in } \Omega.
\end{array}
\end{equation}
Here $\ff \in [L^2(\Omega)]^d$, $g \in L^2_0(\Omega)$, $\mu>0$ is the viscosity coefficient and $\sigma$ a possibly
space dependent coefficient
modelling friction due to the porous medium. Observe that if $\mu=0$
we recover the Darcy model for porous media flow and if $\sigma = 0$
we obtain the classical Stokes' system for creeping incompressible
flow.

The corresponding weak formulation reads: find $(\bu,p) \in V^0\times
Q$ such that:
\[
A_B[(\bu,p),(\bv,q)] = l(\bv),\mbox{ for all } (\bv,q) \in V^0\times
Q.
\]
Here the bilinear forms are given by
\begin{equation}\label{eq:bilin_Brinkman}
A_B[(\bu,p),(\bv,q)]:=a_B(\bu,\bv) - b(p,\bv) + b(q,\bu)
\end{equation}
with
\[
a_B(\bw,\bv) := \int_\Omega \mu \nabla \bw : \nabla \bv + \sigma \bw
\cdot \bv ~\mbox{d}x,
\]
\[
b(q,\bv):= \int_\Omega q \nabla \cdot \bv ~\mbox{d}x
\]
and
\begin{equation}\label{eq:rhsB}
l_B(\bv,q) := \int_\Omega \ff \cdot \bv ~\mbox{d}x +  \int_\Omega g q ~\mbox{d}x.
\end{equation}
By the surjectivity of the divergence operator we may write $\bu =
\bu_0 + \bu_g$ where $\nabla \cdot \bu_g = g$.
Unique existence of  the $\bu_0$ part of the solution is ensured through the application of the
Lax-Milgram lemma in the space  $
H^{div}_0$, where
\[
H^{div}_0:= \{\bv \in V: \nabla \cdot \bv = 0\}.
\]
A unique pressure is then guaranteed by the
Ladyzhenskaya-Babuska-Brezzi condition \cite{BF91}.
\section{The finite element space}\label{sec:FEMconst}
Let $\mathcal{T}_h$ denote a conforming, shape regular tesselation of
$\Omega$ into simplices $T$. We denote the set of faces of the
simplices in $\mathcal{T}$ by $\mathcal{F}$ and the subset of faces that lie on the boundary $\partial \Omega$ by $\mathcal{F}_b$.
We let $X_h$ denote the space of functions in $L^2(\Omega)$ that are
constant on each element,
\[
X_h := \{x \in L^2(\Omega) : x\vert_T \in \mathbb{P}_0(T); \forall T
\in \mathcal{T}_h \}.
\]
The $L^2$-projection on $X_h$, $\pi_0:L^2(\Omega) \mapsto X_h$ is
defined by $(\pi_0 v, x_h)_\Omega = (v,x_h)_{\Omega}$ for all $x_h \in
X_h$. $\pi_0$ satisfies the stability $\|\pi_0 v\|_\Omega \leq
\|v\|_\Omega$ for all $v \in L^2(\Omega)$ and the approximation error estimate
\[
\|\pi_0 v - v\|_\Omega \leq C h |v|_{H^1(\Omega)}, \quad \forall v
\in H^1(\Omega).
\]
We also introduce the $L^2$-projection of the trace of a function
$$\tilde \pi_0 :L^2(\partial \Omega) \mapsto \partial X_h$$
where 
$$
\partial X_h := \{x \in L^2(\partial \Omega) : x\vert_F \in \mathbb{P}_0(F); \forall F
\in \mathcal{F}_b \}
$$
where $\mathcal{F}_b$ is the set of faces in $\mathcal{T}_h$ such that
$F =F \cap \partial \Omega$. We let $W_h$ denote the space of
vectorial piecewise affine functions on $\mathcal{T}_h$,
\[
W_h := \{v \in [H^1(\Omega)]^d: v\vert_T \in  [\mathbb{P}_1(T)]^d; \forall T
\in \mathcal{T}_h \}
\]
and define $Q_h := X_h \cap Q$. It is well known that the space $W_h$
is not robust for nearly incompressible elasticity and that
velocity-pressure space $W_h \times Q_h$ is unstable for
incompressible flow problems. 
To rectify this we will enrich the space with vectorial bubbles on the
faces, following the design in \cite{CH18}, that allows us to remain
conforming in $H^1$, resulting in an extended space, that we will denote $V_h$. The
detailed construction of this space is the topic of the next section.
We then apply $V_h$ in the finite element method for the system of compressible elasticity
and $V_h \times Q_h$ for the Brinkman system. For
the space with built in homogeneous Dirichlet boundary conditions we
write $V_h^0 := V_h \cap [H^1_0(\Omega)]^d$. Observe that by
construction all functions $\bv_h \in V_h^0$ satisfy $\nabla \cdot
\bv_h \in X_h$.

\subsection{Construction of the finite element space $V_h$}
The finite element space is constructed by decomposing every
simplex in subelements. On these subelements face bubbles are
constructed, similar to the face bubbles used in the Bernardi-Raugel
element \cite{BR85}, but in this case they are constructed using piecewise affine
elements. Using the subgrid degrees of freedom similar degrees of
freedoms as in the Bernardi-Raugel element are designed as well. The
upshot here is that the piecewise affine basis functions are designed
so that the divergence restricted to each simplex in the original tesselation is constant. The
pressure space then consists of one constant pressure degree of
freedom per (macro) simplex, allowing for exact imposition of the
divergence free condition. Although the numerical examples in this
work are restricted to the two-dimensional case we below for
completeness also give a detailed description of the construction in
three space dimensions. 

We first treat the 2D case for which our numerical examples are
implemented and then describe how this extends to the three
dimensional case.
Consider a triangular element $T$ twice subdivided. We call the triangle $T$ type I, the first subdivision type II, and the second subdivision type III, cf. Fig \ref{fig:triangles}. The first subdivision is created by joining the centroid of triangle I with its corner nodes. The second subdivision
splits each triangle II by the line joining the centroid of triangle I with the centroid of is neighbouring type I triangle sharing the edge to be split. On the boundary we have a free choice of how to split the edge; we here choose
to split the edge along the line in the direction of the normal to the boundary. 
 On triangles of type I the approximation is piecewise linear with two velocity degrees of freedom in each corner node.
On triangles of type III we add a hierarchical ``bubble'' approximation in the following way. To the node $i$ on the exterior edge $E$ of each triangle of type I is assigned a unit vector $\bnu_i$ along the line $L$ of the split into type III triangles, see Figure \ref{fig:bubbledef}. The unknown in the corresponding edge node $i$ is the vector $a_i\bnu_i$ where $a_i$ is a hierarchical scalar unknown. The centroid--to--centroid nature of the split then ensures continuity of the discrete solution.
In the centroid node the bubble has two velocity components $(u_{xm},u_{ym})$ determined {\em a priori} by setting the divergence $d$ equal (with $a_i=1$) on the triangles sharing node $i$ and  the triangles of type II not
being split by $L$. The divergence is set  by 
\[
d:=\int_E \bnu_i\cdot \bn_E \, ds .
\]
The hierarchical bubble is then piecewise linear on these type II triangles and the type III triangles sharing node $i$. Thus, each edge on triangle I has its own unique hierarchical bubble and the total approximation is the sum of the linear function on type I and the three (vector-valued) bubbles.

A closed form for the velocities defining the bubble associated with an edge can be computed beforehand. With 
the location of the corner, center, and edge nodes according to Fig. \ref{fig:bubbledef}, 
with $A$ the area of triangle $T$, we find
\begin{equation}
{\boldsymbol u}_m=  D({\boldsymbol x}_m - {\boldsymbol x}_o)\\
\end{equation}
where
\[
D := \frac{x_r(y_m - y_l) + x_m(y_l - y_r) + x_l(y_r-y_m)}{2 A \vert{\boldsymbol x}_i-{\boldsymbol x}_m\vert}.
\]
This gives equal divergence $d$ on all subtriangles.

\subsection{The construction of $V_h$ in three space dimensions}
The construction in 3D is analogous to the one in 2D: any given tetrahedron $T$ is decomposed using the Worsey Farin (WF)
split \cite{WF87}, defined as follows. An inpoint is chosen for the tetrahedron, typically (but not necessarily) the center of the inscribed sphere. As inpoint on the (triangular) faces, one chooses (crucially) the point on the line joining the inpoints on the two neighboring tetrahedra. The faces are then split in three subfaces by joining the inpoint to its vertices. The tetrahedron is split in 12 small tetrahedra, three for each face, based on a subface and with summit at the inpoint of the tetrahedron.

The finite element space on the tetrahedron can then be described as the  the space $K(T)$ of continuous $P^1$ vectorfields on the WF split which are divergence free, to which one adds one vector field with constant divergence on $T$, namely ${\boldsymbol x} \mapsto {\boldsymbol x}$. As shown in \cite{CH18} this space has dimension 16. It contains the $P^1$ vectorfields on $T$ (dimension 12), and four bubbles attached to faces (dimension 4). As degrees of freedom one may use vertex values and integrals of normal components on faces.

A face bubble can be defined explicitely for a face $F$, as
follows. We let $\bnu_F$ be the normalized vector parallel to the line
joining the inpoints of the two neighboring tetrahedra of $F$. The
vectorfield on $T$ has value $0$ at vertices of $T$, $\bnu_F$ at the
inpoint of the face $F$, and $0$ at inpoints of the other faces. At the
inpoint of $T$ we determine the vector by the condition that the
divergence of the vector field is the same on all the small tetrahedra
of the WF split and satisfies Stokes' theorem on the three that are
based on $F$. 

\subsection{The Fortin interpolant}

For every $\bu \in V^0$ there exists $\pi_h \bu \in V_h^0$ such that
$\pi_h \bu(x_i) =i_h \bu(x_i)$ in the vertices $x_i$ of type I simplices, where $i_h$ denotes the Cl\'ement
interpolant, and for all $F \in \mathcal{F}$
\[
\int_F \pi_h \bu \cdot \bn_F ~\mbox{d}s = \int_F \bu \cdot \bn_F ~\mbox{d}s.
\]
Note that the interpolant $\pi_h \bu$ satisfies the approximation error estimate
\begin{equation}\label{eq:approx}
\|\pi_h \bu - \bu\|_\Omega
\leq C_1 h |\bu|_{H^1(\Omega)}, \quad h \|\nabla(\pi_h \bu - \bu)\|_\Omega +  \|\pi_h \bu - \bu\|_\Omega
\leq C_2 h^2 |\bu|_{H^2(\Omega)}.
\end{equation}
The proof of the existence of $\pi_h$ is identical to that of the
interpolant for the Bernardi-Raugel element \cite{BR85}. Note that for
functions $\bv \in V$ such that $\bv \cdot \bn = 0$ there holds that
$\tilde \pi_0 (\pi_h \bv)\vert_{\partial \Omega} = 0$.

It follows from this construction that for all $q_h \in Q_h$ and for
all $T \in\mathcal{T}_h$, using the divergence theorem we have
\[
\int_{T}\nabla \cdot \pi_h \bu q_h ~\mbox{d}x = \int_{\partial T}
(\pi_h \bu
\cdot \bn_{\partial T}) q_h ~\mbox{d}s = \int_{\partial T} (\bu
\cdot \bn_{\partial T}) q_h ~\mbox{d}s = \int_{T} \nabla \cdot \bu q_h ~\mbox{d}x=
\int_{T} \pi_0 \nabla \cdot \bu q_h ~\mbox{d}x.
\]
A consequence of the existence of the Fortin interpolant is the
existence of a non-trivial subspace $V_{div}(\bv) \subset V_h$ such
that
\[
V_{div}(\bv) := \{\bv_h \in V_h: \nabla \cdot  \bv_h = \pi_0 \nabla \cdot \bv \}.
\]
As a consequence, for every $q_h \in Q_h$ there exists 
\begin{equation}\label{eq:vp_func}
\srf_q \in
V_h^0 \mbox{ such that }\nabla \cdot \srf_q = q_h \mbox{ and }
\|\srf_q\|_{H^1(\Omega)} \leq C_0 \|q_h\|_\Omega.
\end{equation} 
To see the note that by the surjectivity
of the divergence operator from $V$ to $Q$ for every $q_h \in Q_h$ there
exists $\srf_q \in V$ such that $\nabla \cdot \srf_q = q_h$ and
$\|\srf_q\|_{H^1(\Omega)} \leq C \|q_h\|_\Omega$ and if we
now consider $\pi_h \srf_q \in V_h^0$ we see that $\nabla \cdot \pi_h
\srf_q = \pi_0 \nabla \cdot \srf_q =q_h$ and we conclude that $\srf_q$
may be chosen in $V_h^0$ directly.
\section{Finite element discretization of the model problems}\label{sec:disc}
We consider the finite element spaces $V_h,\,Q_h$ that were defined in the previous section. The
finite element discretization of the problem \eqref{eq:elasticity}
then takes the form: find $\bu_h \in V^0_h$ such that
\begin{equation}\label{eq:FEM_elast}
a_E(\bu_h,\bv_h) = l(\bv_h), \mbox{ for all } \bv_h \in V^0_h,
\end{equation}
where $a_E(\cdot,\cdot)$ and $l(\cdot)$ are defined by
\eqref{eq:bilin_elasticity} and \eqref{eq:rhs}.
The finite element method for the problem \eqref{eq:brinkman} on the
other hand takes the form find $(\bu_h,p_h) \in V_h^0 \times Q_h$ such
that
\begin{equation}\label{eq:FEM_brinkman}
A_B[(\bu_h,p_h),(\bv_h,q_h)] = l_B(\bv_h,q_h),\mbox{ for all } (\bv_h,q_h) \in V_h^0\times
Q_h.
\end{equation}
Both the problem \eqref{eq:FEM_elast} and \eqref{eq:FEM_brinkman}
admit a unique solution by the same arguments as for the continuous
problem. This is also a consequence of the stability estimates that we
derive in the next section.
\section{Stability and error analysis}\label{sec:estimate}
We introduce two triple norms. First for the elasticity system,
\begin{equation}\label{eq:trip_elast}
||| \bv_h |||^2_E := 2 \|\mu^{\frac12} \nabla^s \bv_h\|^2_\Omega + \|\lambda^{\frac12}
\nabla \cdot \bv_h\|^2_\Omega.
\end{equation}
Observe that by Korn's inequality and Poincar\'e's inequality the $E$-seminorm is a norm
on $H^1_0(\Omega)$. Then for the incompressible model we have the
triple norm,
\begin{equation}\label{eq:trip_elast_inc}
||| \bv_h,y_h |||^2_B := \|\mu^{\frac12}  \nabla \bv_h\|^2_\Omega + \|\sigma^{\frac12}
\bv_h\|^2_{\Omega} + \|\nabla \cdot \bv_h\|_\Omega^2 + \|(\mu+\sigma)^{-\frac12}
y_h \|^2_\Omega.
\end{equation}
For the problem \eqref{eq:FEM_elast} Korn's inequality leads to
the
coercivity, there exists $\alpha_E >0$ such that for all $\bv_h \in V_h^0$
\begin{equation}\label{eq:elast_coerciv}
\alpha_E ||| \bv_h |||^2_E \leq a_E(\bv_h,\bv_h).
\end{equation}
For the problem \eqref{eq:FEM_brinkman} we need to prove an inf-sup
condition for stability. 
\begin{prop}(inf-sup stability for the Brinkman problem)\label{prop:infsup}
There exists $\alpha_B$ such that for all $(\bv_h,y_h) \in V_h^0 \times Q_h$ there holds
\[
\alpha_B |||\bv_h, y_h|||_B \leq \sup_{\bw_h,q_h \in (V_h^0\setminus 0)
  \times (Q_h \setminus 0)} \frac{A_B[(\bv_h, y_h),(\bw_h,q_h)]}{|||\bw_h,q_h|||_B}.
\]
\end{prop}
\begin{proof}
First we take $\bw_h = \bv_h$ and $q_h = y_h$ to obtain
\[
\|\mu^{\frac12}  \nabla \bv_h\|^2_\Omega + \|\sigma^{\frac12}
\bv_h\|^2_{\Omega} = A_B[(\bv_h, y_h),(\bw_h,q_h)].
\]
Then we chose $\bw_h = (\mu+\sigma)^{-1} \srf_y$, where $\srf_y$ is
defined by \eqref{eq:vp_func} so that
\[
 (\mu+\sigma)^{-1} \|y_h\|_{\Omega}^2 = A_B[(\bv_h, y_h),(\bw_h,0)] -
 (\mu \nabla \bv_h,\bw_h)_{\Omega}-(\sigma \bv_h,\bw_h)_{\Omega}.
\]
Observing now that
\[
 (\mu \nabla \bv_h,\nabla \bw_h)_{\Omega} \leq \|\mu^{\frac12} \nabla \bv_h\|_\Omega
 \mu^{\frac12} (\mu+\sigma)^{-1} C_0 \|y_h\|_\Omega
\]
and
\[
 (\sigma \bv_h,\bw_h)_{\Omega} \leq \|\sigma^{\frac12} \bv_h\|_\Omega
 \sigma^{\frac12} (\mu+\sigma)^{-1} C_0 \|y_h\|_\Omega
\]
it follows that
\[
 \frac12 (\mu+\sigma)^{-1} \|y_h\|_{\Omega}^2 \leq A_B[(\bv_h, y_h),(\bw_h,0)] -C_0^2(\|\mu^{\frac12} \nabla \bv_h\|_\Omega^2+\|\sigma^{\frac12} \bv_h\|_\Omega^2).
\]
Taking $\bw_h = \bv_h+(2 C_0)^{-1} (\mu+\sigma)^{-1} \srf_y$ and
$q_h=y_h + \nabla \cdot \bv_h$ we conclude that
\begin{multline*}
\min\left(\frac12, \frac{1}{2C_0}\right) ||| \bv_h,y_h|||^2_B \leq \frac12 \|\mu^{\frac12}  \nabla \bv_h\|^2_\Omega + \frac12 \|\sigma^{-\frac12}
\bv_h\|^2_{\Omega} + \frac{1}{2C_0} (\mu+\sigma)^{-1}
\|y_h\|_{\Omega}^2 \\
\leq A_B[(\bv_h, y_h),(\bw_h,q_h)] 
\end{multline*}
To finish the proof note that
\begin{multline*}
||| \bw_h,q_h|||_B \leq  ||| \bv_h,y_h|||_B +  ||| (2 C_0)^{-1}
(\mu+\sigma)^{-1} \srf_y,0|||_B \\
\leq  ||| \bv_h,y_h|||_B +   (2
C_0)^{-1}\mu^{\frac12}(\mu+\sigma)^{-1} C_0 \|y_h\|_\Omega + \|\nabla
\cdot \bv_h\|_\Omega\leq C  ||| \bv_h,y_h|||_B.
\end{multline*}
\end{proof}
Using the stability estimates we may now prove error estimates for the
approximations of \eqref{eq:FEM_elast} and \eqref{eq:FEM_brinkman}.
\begin{prop}
Let $\bu$ be the solution of \eqref{eq:elasticity} and $\bu_h$ the solution
of \eqref{eq:FEM_elast} then
\[
\|\mu^{\frac12} \nabla (\bu - \bu_h)\|_\Omega + \|\lambda^{\frac12}
\left( \pi_0 \nabla \cdot \bu - \nabla \cdot \bu_h \right)\|_\Omega \leq C \inf_{\bv_h \in V_{div}(\bu)} \|\mu^{\frac12} \nabla
(\bu - \bv_h)\|_\Omega
\]
and
\[
\|\mu^{\frac12} \nabla (\bu - \bu_h)\|_\Omega +
\|\lambda^{\frac12}\left(\nabla \cdot \bu - \nabla \cdot
  \bu_h\right)\|_\Omega  \leq C h (
\mu^{\frac12}\|\bu\|_{H^2(\Omega)} + \lambda^{\frac12} \|\nabla \cdot \bu\|_{H^1(\Omega)} )\leq C_E h \|\ff\|_\Omega.
\]
where $C_E$ is independent of $\lambda$.
\end{prop}
\begin{proof}
Let $\be_h := \bu_h - \bv_h$, with $\bv_h \in  V_{div}(\bu)$. Note that by adding and subtracting
$\bw_h$ and using the triangle inequality and Korn's inequality we have
\[
\|\mu^{\frac12} \nabla (\bu - \bu_h)\|_\Omega + \|\lambda^{\frac12}\left(
\pi_0 \nabla \cdot \bu - \nabla \cdot \bu_h\right)\|_\Omega \leq \|\mu^{\frac12} \nabla
(\bu - \bv_h)\|_\Omega + |||\be_h|||_E.
\]
For the second term we apply the coercivity \eqref{eq:elast_coerciv},
followed by Galerkin orthogonality 
\[
a_E(\bu - \bu_h,\bw_h) = 0 \mbox{ for all } \bw_h \in V_h^0
\]
to obtain
\[
\alpha_E |||\be_h|||^2_E \leq a_E(\be_h,\be_h) = a_E(\bu-\bv_h,\be_h).
\]
Noting that 
\begin{equation}\label{eq:cont_div}
(\lambda \nabla \cdot (\bu- \bv_h),\nabla \cdot
\be_h)_\Omega = (\lambda (\nabla \cdot \bu- \pi_0\nabla \cdot  \bu),\nabla \cdot
\be_h)_\Omega = 0
\end{equation}
we may write
\begin{equation}\label{eq:cont_eps}
\alpha_E |||\be_h|||^2_E \leq (2 \mu \nabla^s (\bu- \bv_h),
\nabla^s \be_h)_\Omega \leq 2 \|\mu^{\frac12} \nabla (\bu-
\bv_h)\|_\Omega |||\be_h|||_E,
\end{equation}
which proves the first claim.

The second claim is immediate, taking $\bv_h = \pi_h \bu$ and using the approximation properties of
$\pi_h$, \eqref{eq:approx} and the regularity bound
\eqref{eq:elasticity_regularity}. To show that the constant $C_E$is
independent of $\lambda$ observe that $\lambda^{\frac12} \|\nabla \cdot
\bu\|_{H^1(\Omega)} \leq \max( c \mu^{\frac12}|\bu|_{H^2(\Omega)},
\lambda \|\nabla \cdot
\bu\|_{H^1(\Omega)} )$.
\end{proof}
\begin{prop}\label{prop:apriori_brink}
Let $(\bu,p) \in V \times Q$ be the solution to \eqref{eq:brinkman}, with $\mu>0$, $\sigma \ge 0$
and $(\bu_h,p_h)$ the solution to \eqref{eq:FEM_brinkman}. Then there
holds
\[
|||\bu-\bu_h,\pi_0 p-p_h|||_B \leq C \inf_{\bv_h \in V_{div}^g}  |||\bu-\bv_h,0|||_B
\]
where $V_{div}^g := \{\bv \in V_h^0: \nabla \cdot \bv = \pi_0 g \}$
and
\[
|||\bu-\bu_h,\pi_0 p-p_h|||_B \leq C (h \mu^{\frac12}
|u|_{H^2(\Omega)}+ \min (C_1 h 
\sigma^{\frac12} |u|_{H^1(\Omega)}, C_2 h^2
\sigma^{\frac12} |u|_{H^2(\Omega)}).
\]
\end{prop}
\begin{proof}
We introduce, as before, discrete errors $\be_h := \bu_h - 
\bv_h$, with $\bv_h \in V_{div}^g$ and $\eta_h = \pi_0 p - p_h$. Using the triangle inequality we
see that
\[
|||\bu-\bu_h,0|||_B \leq |||\bu-\bv_h,0|||_B + |||\be_h,\eta_h|||_B.
\]
For the second term in the right hand side we apply the stability of
Proposition \ref{prop:infsup} to obtain
\[
|||\be_h,\eta_h|||_B \leq \sup_{\bw_h,q_h \in (V_h^0\setminus 0)
  \times (Q_h \setminus 0)} \frac{A_B[(\be_h, \eta_h),(\bw_h,q_h)]}{|||\bw_h,q_h|||_B}.
\]
using Galerkin orthogonality
\[
\]
we have
\begin{equation}\label{eq:Agalortho}
A_B[(\be_h, \eta_h),(\bw_h,q_h)] = A_B[(\bu-\bv_h,p-\pi_0 p),(\bw_h,q_h)].
\end{equation}
Observe that by construction we have
\[
b(q_h,\bu- \bv_h) = 0 \mbox{ and } b(p - \pi_0 p,\bw_h) = 0.
\]
The only remaining term in the right hand side of \eqref{eq:Agalortho}
is bounded using the Cauchy--Schwarz inequality,
\[
a_B(\bu-\pi_h \bu,\bw_h) \leq |||\bu-\bv_h,0|||_B |||\bw_h,q_h|||_B.
\]
This proves the first claim and the second follows as before taking
$\bv_h = \pi_h \bu \in V_{div}^g$ and using the
approximation properties of the Fortin interpolant $\pi_h$ \eqref{eq:approx}.
\end{proof}
Since we have imposed the boundary conditions strongly above we can
not take $\mu=0$ in the Brinkman model corresponding to the case of
the Darcy equations. In order to make this limit feasible we will now
discuss weak imposition of boundary conditions using Nitsche's method.
\section{Weakly imposed boundary conditions, Nitsche's method}\label{sec:nitsche}
Here we will discuss how to impose non-penetration conditions on the
space $V_h$ as one wishes to do in the case of zero-traction boundary
conditions in elasticity and how to relax the no-slip condition when
$\mu \rightarrow 0$ for the Brinkman model. Therefore we here
propose Nitsche methods for the imposition of boundary conditions that
preserve the locking free character for elasticity and are robust in
the limit of pure porous media flow for the Brinkman model. 
\subsection{Zero traction conditions for linear elasticity}
Consider first the elasticity problem \eqref{eq:elasticity}, with the
boundary decomposed in $\partial \Omega:=\overline{\partial \Omega_D}
\cup \overline{\partial \Omega_N}$ where $\partial \Omega_D$ and
$\partial \Omega_N$ each consists of a set of entire polyhedral
faces. We assume that 
\begin{equation}\label{eq:zero_traction_bc}
\bft \bu =\bg_D \mbox{ on } \partial \Omega_D \mbox{ and }\bu
\cdot \bn = g_N \mbox{ on } \partial \Omega \mbox{ and } \bft (\sigma(\bu)
\bn) = 0 \mbox{ on }\partial \Omega_N.
\end{equation}
Here the tangential projection is defined by $\bft := \mathbb{I}- \bn \otimes \bn$.
The Nitsche formulation then
takes the form: Find $\bu_h \in V_h$ such that
\begin{equation}\label{eq:FEM_elasticity}
A_{E,h}(\bu_h,\bv_h) = L(\bv_h) 
\end{equation}
with
\[
A_{E,h}(\bu_h,\bv_h):=a_E(\bu_h,\bv_h) - c(\bu_h,\bv_h)-c(\bv_h,\bu_h)+s(\bu_h,\bv_h)
\]
and
\[
 L(\bv_h) = l(\bv_h) + l_c(\bv_h)
\]
where
\[
c(\bu_h,\bv_h):= (\bn \cdot(\bsigma(\bu_h) \bn), \bv_h \cdot \bn)_{\partial \Omega}
+ (\bft \cdot(\bsigma(\bu_h) \bn), \bft \bv_h)_{\partial \Omega_D}
\]
\[
s(\bu_h,\bv_h) :=
(\gamma /h (\mu  + \lambda  \tilde \pi_0) ~\bu_h \cdot \bn,  \bv_h \cdot \bn)_{\partial \Omega} +(\gamma \mu/h ~\bft\bu_h, \bft\bv_h)_{\partial
  \Omega_D}
\]
and
\[
l_c(\bv_h) = (g_N, \gamma /h (\mu+\lambda \tilde \pi_0 )  ~\bv_h \cdot \bn  -  \bn
\cdot(\bsigma(\bv_h) \bn) )_{\partial \Omega} + (\bg_T, \gamma \mu/h~
\bft\bv_h -\bft
\cdot(\bsigma(\bv_h) \bn))_{\partial \Omega_D}.
\]
Observe that the projection $\tilde \pi_0$ in the boundary penalty of the
normal component is necessary to
avoid locking. 

We define the stabilization semi-norm by
\[
|\bv_h|_s := s(\bv_h,\bv_h)^{\frac12}
\]
and the following augmented energy norm defined on $H^1(\Omega)$
\[
|||\bv_h|||_{E,h}^2 := |||\bv_h|||_{E}^2 + |\bv_h|^2_s.
\]
We recall that $|||\cdot |||_{E,h}$ is a norm by Korn's inequality and
Poincar\'e's inequality.
We recall the trace inequalities
\begin{equation}\label{eq:trace1}
\|\bv\|_{\partial T} \leq C_T(h^{-\frac12} \|\bv\|_{T} + h^{\frac12}
\|\nabla \bv \|_{T}) \quad \forall T
\mbox{ and } \bv \in H^1(T)
\end{equation}
and
\begin{equation}\label{eq:trace2}
\|\bv_h\|_{\partial T} \leq C_T h^{-\frac12} \|\bv_h\|_{T} \quad \forall T
\mbox{ and } \bv_h \in V_h.
\end{equation}
Using these inequalities it is straightforward to prove the following
approximation estimate in the norm $|||\cdot|||_{E,h}$ and a
bound on the form $c$.
\begin{lem}\label{lem:approx_elast}
The following approximation inequality holds
\begin{equation}\label{eq:approx_elast}
|||\bu - \pi_h \bu |||_{E,h} \leq C h (\mu^{\frac12} |\bu|_{H^2(\Omega)}
+ \lambda^{\frac12} |\nabla \cdot \bu|_{H^1(\Omega)}).
\end{equation}
\end{lem}
\begin{proof}
The inequality 
\[
|||\bu - \pi_h \bu |||_{E} \leq C h (\mu^{\frac12} |\bu|_{H^2(\Omega)}
+ \lambda^{\frac12} |\nabla \cdot \bu|_{H^1(\Omega)}).
\]
is immediate by the commuting
property and approximation 
properties of the Fortin interpolant.
Considering the stabilization part we see that using \eqref{eq:trace1} on
each boundary face followed by the approximation (\ref{eq:approx}),
\[
(\mu/h )^{\frac12} \|(\bu - \pi_h \bu)\cdot \bn\|_{\partial \Omega} \leq C h \mu^{\frac12} |\bu|_{H^2(\Omega)}.
\]
Using the definition of $\pi_h$ we see that $\tilde \pi_0 \pi_h \bu =
\tilde \pi_0 \bu$ and therefore
\[
(\lambda/h )^{\frac12} \|\tilde \pi_0 (\bu - \pi_h \bu)\cdot
\bn\|_{\partial \Omega} = 0.
\]
This last property is necessary to prove that the method is locking free.
\end{proof}
\begin{equation}\label{eq:approx_Brink}
\end{equation}
\begin{lem}\label{eq:nit_trace}
For $\epsilon>0$ there holds
\begin{equation}
c(\bu_h,\bu_h) \leq \epsilon|||\bu_h|||^2_E + \epsilon^{-1} C_T^2 
\gamma^{-1} |\bu_h|_s^2.
\end{equation}
\end{lem}
\begin{proof}
This proof follows the ideas of \cite{HL03}, we include it here for completeness.
First we note that
\[
c(\bu_h,\bu_h)  =  (2 \mu \bn \cdot \nabla^s \bu_h \bn+ \lambda \nabla \cdot
\bu_h, \bu_h \cdot \bn)_{\partial \Omega}
+ (2 \mu \bft \cdot \nabla^s \bu_h \bn, \bu_h \cdot \bft)_{\partial
  \Omega_D}.
\]
Since for $F \in \mathcal{F}_b$, $\nabla \cdot \bu_h\vert_F \in \mathbb{P}_0(F)$ there holds
\[
(\lambda \nabla \cdot
\bu_h, \bu_h \cdot \bn)_{\partial \Omega} = (\lambda \nabla \cdot
\bu_h, \tilde \pi_0 \bu_h \cdot \bn)_{\partial \Omega}.
\]
Applying the Cauchy--Schwarz inequality followed by the trace inequality
\eqref{eq:trace2} we see that for all $\epsilon>0$,
\[
(2 \mu \bn \cdot \nabla^s \bu_h \bn, \bu_h \cdot \bn)_{\partial
  \Omega} \leq 2 C_T \|\mu^{\frac12} \nabla^s \bu_h\|_\Omega
\|\mu^{\frac12} h^{-\frac12} \bu_h \cdot \bn\|_{\partial \Omega} \leq \epsilon \|\mu^{\frac12} \nabla^s
\bu_h\|^2_\Omega+  C_T^2 \epsilon^{-1}\|\mu^{\frac12} h^{-\frac12} \bu_h \cdot \bn\|_{\partial \Omega}^2
\]
\[
  (2 \bft \cdot \mu \nabla^s \bu_h \bn, \bu_h \cdot \bft)_{\partial
  \Omega_D}\leq  2 C_T \|\mu^{\frac12} \nabla^s \bu_h\|_\Omega
\|\mu^{\frac12} h^{-\frac12} \bu_h \cdot \bft\|_{\partial \Omega_D} \leq \epsilon \|\mu^{\frac12} \nabla^s
\bu_h\|^2_\Omega+  C_T^2 \epsilon^{-1}\| \mu^{\frac12} h^{-\frac12}\bu_h \cdot \bft\|_{\partial \Omega_D}^2
\]
and
\begin{multline*}
(\lambda \nabla \cdot
\bu_h, \tilde \pi_0 \bu_h \cdot \bn)_{\partial \Omega}\leq  C_T
\|\lambda^{\frac12} \nabla \cdot \bu_h\|_\Omega
\gamma^{-\frac12} \|\lambda^{\frac12} h^{-\frac12}\tilde \pi_0 \bu_h \cdot
\bn\|_{\partial \Omega} \\
\leq \epsilon \|\lambda^{\frac12} \nabla
\cdot \bu_h\|^2_\Omega+ C_T^2 4^{-1} \epsilon^{-1}\|\lambda^{\frac12} h^{-\frac12} \tilde \pi_0 \bu_h \cdot
\bn\|_{\partial \Omega}^2.
\end{multline*}
Summing up the different contributions and observing that
\[
\|\mu^{\frac12} h^{-\frac12} \bu_h \cdot \bn\|_{\partial \Omega}^2+\|\lambda^{\frac12} h^{-\frac12} \tilde \pi_0 \bu_h \cdot
\bn\|_{\partial \Omega}^2+\| \mu^{\frac12} h^{-\frac12}\bu_h \cdot
\bft\|_{\partial \Omega_D}^2 \leq \gamma^{-1} |\bu_h|_s^2
\]
we see that
\[
c(\bu_h,\bu_h) \leq \epsilon (2\|\mu^{\frac12} \nabla^s
\bu_h\|^2_\Omega+ \|\lambda^{\frac12} \nabla
\cdot \bu_h\|^2_\Omega) + C_T^2 \epsilon^{-1}\gamma^{-1} |\bu_h|_s^2.
\]
This proves the claim.
\end{proof}
\begin{lem}\label{nit_coerciv}
Assume that $\gamma \ge 4 C_T \epsilon^{-1}$, with $0<\epsilon<1$,
then there exists $\alpha>0$ such that for all $\bv_h \in V_h$ there holds,
\[
\alpha |||\bv_h|||_{E,h}^2 \leq A_{E,h}(\bv_h,\bv_h).
\]
For the choice $\gamma=16 C_T^2$, $\alpha=\tfrac12$.
\end{lem}
\begin{proof}
By definition 
\[
A_{E,h}(\bv_h,\bv_h) \ge 2 \|\mu^{\frac12} \nabla^s \bv_h\|_\Omega^2 +
\|\lambda^{\frac12} \nabla \cdot \bv_h\|_\Omega^2 + |\bv_h|^2_s - 2
c(\bv_h,\bv_h) \ge |||\bv_h|||^2_E + |\bv_h|^2_s - 2
c(\bv_h,\bv_h).
\]
Using the result of Lemma \ref{eq:nit_trace} we see that
\begin{multline*}
A_{E,h}(\bv_h,\bv_h) \ge  |||\bv_h|||^2_E + |\bv_h|^2_s - \epsilon
|||\bv_h|||^2_E - 4 \epsilon^{-1} C_T^2 
\gamma^{-1} |\bv_h|_s^2 \\
= (1-\epsilon) |||\bv_h|||^2_E + (1-4 \epsilon^{-1} C_T^2 
\gamma^{-1}) |\bv_h|_s^2.
\end{multline*}
Taking $0<\epsilon<1$ and $\gamma \ge 4 C_T^2 \epsilon^{-1}$ proves the claim.
For the particular choice $\epsilon=1/2$ and $\gamma=16 C_T^2$ we see that
\[
A_{E,h}(\bv_h,\bv_h) \ge  \frac12 |||\bv_h|||^2_{E,h} .
\]
\end{proof}
\begin{prop}
Let $\bu$ be the solution of \eqref{eq:elasticity} with the boundary conditions
\eqref{eq:zero_traction_bc} and $\bu_h$ the solution of
\eqref{eq:FEM_elasticity}, then there holds
\[
||| \bu - \bu_h|||_{E,h} \leq C h \|\ff\|_\Omega
\]
where the constant $C$ is independent of $\lambda$.
\end{prop}
\begin{proof}
First note that by the triangle inequality there holds
\[
||| \bu - \bu_h|||_{E,h} \leq ||| \bu - \pi_h \bu|||_{E,h} 
+ ||| \pi_h \bu - \bu_h|||_{E,h}.
\]
Using the coercivity of Lemma \ref{nit_coerciv} we have, with $\be_h:= \pi_h \bu - \bu_h$
\[
\frac12 ||| \be_h|||_{E,h}^2 \leq A_{E,h}(\be_h,\be_h).
\]
Using now the consistency of $A_{E,h}$ we see that
\[
\frac12 ||| \be_h|||_{E,h}^2 \leq A_{E,h}(\pi_h \bu - \bu,\be_h).
\]
We also have the following continuity of the form $A_{E,h}$,
\[
A_{E,h}(\pi_h \bu - \bu,\be_h) \leq C |||\be_h|||_{E,h}(|||\pi_h \bu -
\bu|||_{E,h} + h^{\frac12}\|\mu^{\frac12} \nabla^s (\pi_h \bu -
\bu)\|_{\partial \Omega} +
h^{\frac12}\|\lambda/\mu^{\frac12} \nabla \cdot (\pi_h \bu -
\bu) \|_{\partial \Omega}),
\]
here we used 
the Cauchy-Schwarz inequality termwise and, for the terms with a factor $\lambda$,
the relations
\[
\lambda (\nabla \cdot \be_h + h^{-1} \tilde \pi_0 \be_h, \pi_h \bu -
\bu)_{\partial \Omega}  = 0
\]
and
\[
\lambda (\nabla \cdot (\pi_h \bu -
\bu),\be_h \cdot \bn)_{\partial \Omega} \leq \lambda \mu^{-\frac12}
h^{\frac12} \|\nabla \cdot (\pi_h \bu -
\bu)\|_{\partial \Omega} \mu^{\frac12} h^{-\frac12} \|\be_h \cdot \bn\|_{\partial \Omega}.
\]
It follows that
\[
\frac12 ||| \be_h|||_{E,h} \leq (|||\pi_h \bu -
\bu|||_{E,h} + h^{\frac12}\|\mu^{\frac12} \nabla^s (\pi_h \bu -
\bu)\|_{\partial \Omega} +
h^{\frac12}\|\lambda/\mu^{\frac12} \nabla \cdot (\pi_h \bu -
\bu) \|_{\partial \Omega}).
\]
and as a consequence
\[
 ||| \bu - \bu_h |||_{E,h} \leq C(|||\pi_h \bu -
\bu|||_{E,h} + h^{\frac12}\|\mu^{\frac12} \nabla^s (\pi_h \bu -
\bu)\|_{\partial \Omega} +
h^{\frac12}\mu^{-\frac12} \|\lambda \nabla \cdot (\pi_h \bu -
\bu) \|_{\partial \Omega}).
\]
The error estimate is concluded by the approximation result of Lemma
\ref{lem:approx_elast} and the inequality \eqref{eq:trace1} by which
\begin{multline}\label{eq:bound_approx}
h^{\frac12}\|\mu^{\frac12} \nabla^s (\pi_h \bu -
\bu)\|_{\partial \Omega} +
h^{\frac12} \mu^{-\frac12} \|\lambda^{\frac12} \nabla \cdot (\pi_h \bu -
\bu) \|_{\partial \Omega}\leq C (\|\mu^{\frac12} \nabla^s (\pi_h \bu -
\bu)\|_{\Omega} + \mu^{-\frac12} \|\lambda\nabla \cdot (\pi_h \bu -
\bu) \|_{\Omega})\\
+ C h (\mu^{\frac12}  |\bu|_{H^2(\Omega)} + \mu^{-\frac12} \lambda |\nabla \cdot \bu|_{H^1(\Omega)}),
\end{multline}
followed by approximation. This leads to
\[
||| \bu - \bu_h|||_{E,h} \leq C \mu^{-\frac12} h
(\mu |\bu|_{H^2(\Omega)} + \lambda |\nabla \cdot \bu|_{H^1(\Omega)}) \leq
C h \|\ff\|_\Omega,
\]
where $C$ depends on $\mu$ but not on $\lambda$. The second inequality
is a consequence of the elliptic regularity \eqref{eq:elasticity_regularity}.
\end{proof}
\subsection{Zero viscosity limit for the Brinkman problem}
We now consider the problem \eqref{eq:brinkman}, but instead of
imposing Dirichlet boundary conditions strongly we here consider using 
Nitsche's method on the tangential component. The Dirichlet condition
on the normal component is still imposed strongly. This way the method can
handle all values of the viscosity, also $\mu=0$. To fix the ideas we
assume that $\sigma>0$ and $\mu \ge 0$ in \eqref{eq:brinkman}. If
$\mu=0$ we only impose the boundary condition on the normal component
\begin{equation}\label{eq:no_inflow}
\bu_h \cdot \bn\vert_{\partial \Omega} = 0.
\end{equation}
We see that the finite element solution will then be found in a
subspace of
$$
V^0_{\bn} := \{ \bv \in V: (\bv \cdot \bn)\vert_{\partial \Omega}  = 0\}
$$
instead of $V^0$.
To impose this condition strongly on the discrete solution we introduce the
space 
\[
V^0_{\bn,h} := \{ \bv \in V_h: \bv \cdot \bn = 0\}.
\]
This space can easily be constructed on polyhedral domains, by setting both
the boundary bubble degrees of freedom and the normal component of the
nodal degrees of freedom to zero. The Dirichlet condition
on the tangential component will then be imposed using Nitsche's
method \cite{FS95}.

This time the Nitsche formulation 
takes the form: Find $(\bu_h, p_h) \in V^0_{\bn,h} \times Q_h$ such that
\begin{equation}\label{eq:FEM_Brinkman_nitsche}
A_{B,h}(\bu_h,\bv_h) = l_B(\bv_h,q_h), \quad \forall (\bv_h, q_h) \in V^0_{\bn,h} \times Q_h
\end{equation}
with
\[
A_{B,h}(\bu_h,\bv_h):=A_B(\bu_h,\bv_h) - m(\bu_h,\bv_h)-m(\bv_h,\bu_h)+s(\bu_h,\bv_h)
\]
where
\[
m(\bu_h,\bv_h):= (\bft  \sigma(\bu_h,p_h)  \bn, \bft
\bv_h)_{\partial \Omega} = (\bft  \mu \nabla \bu_h\bn, \bft
\bv_h)_{\partial \Omega}
\]
with $\sigma(\bu,p) := \mu \nabla \bu - p \mathbb{I}$
and
\[
s(\bu_h,\bv_h) :=(\gamma/h \mu ~ \bft \bu_h , ~ \bft \bv_h)_{\partial
  \Omega}.
\]
For the analysis of the Nitsche conditions we define the
triple norm
\begin{equation}\label{def:3B}
|||\bv_h,y_h|||_{B,h}^2 := |||\bv_h,y_h|||_{B}^2 + |\bv_h|^2_s.
\end{equation}
As noted in section \ref{sec:FEMconst} there exists an interpolant $\pi_{h,\bn}:V_{\bn}^0 \mapsto V^0_{\bn,h} $ with the same commutation and approximation properties as
$\pi_h$ in
\eqref{eq:approx}, with some abuse of notation we drop the subscript
$\bn$ below. In particular it is straightforward to show, using the same arguments as in Lemma \ref{eq:approx_elast},  that
the following Lemma holds.
\begin{lem}\label{lem:approx1}
Let $\bu \in V_{\bn}^0$ then there holds
\[
|||\bu - \pi_h \bu, 0 |||_{B,h} \leq C h (\mu^{\frac12} |\bu|_{H^2(\Omega)}
+ \sigma^{\frac12} |\bu|_{H^1(\Omega)})
\]
\end{lem}
\begin{proof}
The proof is identical to that of Lemma \ref{lem:approx_elast}.
\end{proof}
\begin{prop}\label{prop:infsup_nit}
There exists $\alpha_B$, such that, assuming 
$\gamma$ large enough, then for all $(\bv_h,y_h) \in V_{\bn,h}^0 \times Q_h$ there holds
\[
\alpha_B |||\bv_h, y_h|||_{B,h} \leq \sup_{\bw_h,q_h \in (V_{\bn,h}^0\setminus 0)
  \times Q_h )} \frac{A_{B,h}[(\bv_h, y_h),(\bw_h,q_h)]}{|||\bw_h,q_h|||_B}.
\]
\end{prop}
\begin{proof}
First we take $\bw_h = \bv_h$ and $q_h = y_h$ to obtain
\[
\|\mu^{\frac12}  \nabla \bv_h\|^2_\Omega + \|\sigma^{-\frac12}
\bv_h\|^2_{\Omega}+|\bv_h|_s^2 -2 m(\bv_h,\bv_h) = A_{B,h}[(\bv_h, y_h),(\bw_h,q_h)].
\]
Following the same arguments as in Lemma \ref{eq:nit_trace} we see
that
\[
 m(\bv_h,\bv_h) \leq \epsilon |||\bv_h,0|||^2_B + C^2_T \epsilon^{-1} \gamma^{-1} |\bv_h|_s^2.
\]
It follows that taking  $0<\epsilon<1/2$ and $\gamma \ge  4 C^2_T/\epsilon$ we have
\[
\frac12 (1-2 \epsilon) ||| \bv_h,0|||_{B,h}^2 \leq A_{B,h}[(\bv_h, y_h),(\bw_h,q_h)].
\]
Then we chose $\bw_h = (\mu+\sigma)^{-1} \srf_y$, with $\srf_y$ as in \eqref{eq:vp_func}. 
\begin{multline*}
 (\mu+\sigma)^{-1} \|y_h\|_{\Omega}^2 = A_B[(\bv_h, y_h),(\bw_h,0)] -
 (\mu \nabla \bv_h,\nabla \bw_h)_{\Omega}-(\sigma
 \bv_h,\bw_h)_{\Omega}\\+m(\bw_h,\bv_h) +m(\bv_h,\bw_h)-\gamma (\mu/h\,
 \bft \bv_h,\bft \bw_h)_{\partial \Omega}.
\end{multline*}
The second and the third terms on the right hand side are handled as
in Proposition \ref{prop:infsup}.
\begin{multline}
(\mu \nabla \bv_h,\nabla \bw_h)_{\Omega}+(\sigma
 \bv_h,\bw_h)_{\Omega} \leq C_0^{2} \|\mu^{\frac12}\nabla
 \bv_h\|_\Omega^2+C_0^{2} \|\sigma^{\frac12}  \bv_h\|_\Omega^2 +
 \frac14 C_0^{-2}  (\mu +
 \sigma) \|\bw_h\|_{H^1(\Omega)}^2\\
\leq C_0^2 \|\mu^{\frac12}\nabla  \bv_h\|_\Omega^2+C_0^2\|\sigma^{\frac12}
\bv_h\|_\Omega^2 + \frac14 (\mu +
 \sigma)^{-1} \|y_h\|_{H^1(\Omega)}^2
\end{multline}

Using the Cauchy-Schwarz inequality,
the trace inequality \eqref{eq:trace2} and the fact that $\srf_y \in V_h^0$ we see that
\begin{multline*}
m(\bw_h,\bv_h) +m(\bv_h,\bw_h) -\gamma(\mu/h\,
 \bft \bv_h,\bft \bw_h)_{\partial \Omega} \leq C_T 
  \|\mu^{\frac12} \nabla\bw_h\|_\Omega \gamma^{-\frac12} |
  \bv_h|_s \\
\leq \frac14 (\sigma+\mu)^{-1} \|y_h\|^2_\Omega + (C_T C_0)^2 \gamma^{-1} |
  \bv_h|^2_s
\end{multline*}
Summing the above bounds 
it follows that,
\[
 \frac12 (\mu+\sigma)^{-1} \|y_h\|_{\Omega}^2 \leq A_B[(\bv_h,
 y_h),(\bw_h,0)] + C_0^2 |||\bv_h,0|||_{B,h}^2,
\]
where we used that $(C_T C_0)^2 \gamma^{-1} \leq C_0^2 \epsilon/4 \leq C_0^2$.
Taking $\bw_h = \bv_h+(2 C_0)^{-2} (\mu+\sigma)^{-1} \srf_y$ and
$q_h=y_h + \nabla \cdot \bv_h$ we deduce that
\begin{equation*}
\left(\frac{1}{4}- 2 \epsilon\right)|||\bv_h,0 |||^2_{B,h} + \frac{1}{8C_0^2} ||| 0,y_h |||^2_{B,h}
\leq A_B[(\bv_h, y_h),(\bw_h,q_h)].
\end{equation*}
Let now $\epsilon=\frac{1}{16}$ then for $\alpha = 1/8\min (1,
C_0^{-2}) >0$ there holds,
\[
\alpha ||| \bv_h,y_h |||^2_{B,h} \leq A_B[(\bv_h, y_h),(\bw_h,q_h)].
\]
To finish the proof note that, as before,
\begin{multline*}
||| \bw_h,q_h|||_{B,h} \leq  ||| \bv_h,y_h|||_{B,h} +  ||| (2 C_0)^{-2}
(\mu+\sigma)^{-1} \srf_y,0|||_{B,h} \\
\leq  ||| \bv_h,y_h|||_{B,h} +   (2
C_0)^{-2}\mu^{\frac12}(\mu+\sigma)^{-1} C_0 \|y_h\|_\Omega\leq C_B  ||| \bv_h,y_h|||_{B,h},
\end{multline*}
where $C_B$ is independent of $\mu$ and $\sigma$, but not of $C_0$.
The inequality then holds with $\alpha_B = \alpha/C_B$.
\end{proof}
Optimal a priori estimates follow using the stability of Proposition
\ref{prop:infsup_nit}, consistency and continuity.
\begin{prop}\label{prop:apriori_brinkman_nit}
Under the hypothesis of Proposition \ref{prop:infsup_nit},
let $(\bu,p) \in V^0_{\bn} \times Q$ be the solution to \eqref{eq:brinkman},
with either $\mu>0$ and $\sigma \ge 0$ or  $\mu \ge 0$ and $\sigma > 0$
and $(\bu_h,p_h) \in V^0_{\bn,h} \times Q_h$ the solution to \eqref{eq:FEM_brinkman}. Then there
holds
\[
|||\bu-\bu_h,p-p_h|||_{B,h} \leq C h (\mu^{\frac12}
|u|_{H^2(\Omega)}+ 
\sigma^{\frac12} |u|_{H^1(\Omega)}).
\]
\end{prop}
\begin{proof}
We introduce, as before, the discrete errors $\be_h := \bu_h - \pi_h
\bu$ and $\eta_h = \pi_0 p - p_h$. Using the triangle inequality we
see that
\[
|||\bu-\bu_h,0|||_{B,h} \leq |||\bu-\pi_h \bu,0|||_{B,h} + |||\be_h,\eta_h|||_{B,h}.
\]
For the second term in the right hand side we apply the stability of
Proposition \eqref{prop:infsup_nit} to obtain
\[
\alpha_B |||\be_h,\eta_h|||_{B,h} \leq \sup_{\bw_h,q_h \in (V_h\setminus 0)
  \times (Q_h \setminus 0)} \frac{A_{B,h}[(\be_h, \eta_h),(\bw_h,q_h)]}{|||\bw_h,q_h|||_{B,h}}.
\]
using Galerkin orthogonality
we have
\[
A_{B,h}[(\be_h, \eta_h),(\bw_h,q_h)] = A_{B,h}[(\bu-\pi_h \bu,p-\pi_0 p),(\bw_h,q_h)].
\]
The form $A_B$ is handled as in Proposition
\ref{prop:apriori_brink}. Using the orthogonality properties of the
$\pi_h$ and $\pi_0$ we see that
\[
A_{B,h}[(\bu-\pi_h \bu,p-\pi_0 p),(\bw_h,q_h)] \leq |||\bu-\pi_h \bu,\eta_h|||_{B,h} |||\be_h,0|||_{B,h} 
+|m(\bw_h,\bu-\pi_h \bu)|+|m(\bu-\pi_h \bu,\bw_h)|
\]
For the Nitsche terms we see that using the Cauchy-Schwarz inequality
followed by the trace inequality \eqref{eq:trace1} and the
approximation of Lemma \ref{lem:approx1}
\[
m(\bw_h,\bu-\pi_h \bu) +m(\bu-\pi_h \bu,\bw_h) \leq C ||| \bw_h,0|||_{B,h} \mu^{\frac12} h |\bu|_{H^2(\Omega)}
\]
where we also used an argument similar to \eqref{eq:bound_approx} to
obtain the bound $\|\mu^{\frac12} \nabla (\bu-\pi_h \bu)\|_{\partial
  \Omega} \leq C  \mu^{\frac12} h |\bu|_{H^2(\Omega)}$.
The stabilization term is bounded by applying the Cauchy-Schwarz inequality 
\[
s(\bu-\pi_h \bu, \bw_h) \leq |\bu-\pi_h \bu|_s | \bw_h|_s \leq |||\bu-\pi_h \bu, 0|||_{B,h} |||\bw_h, 0|||_{B,h} .
\]
We conclude that 
\[
\alpha_B |||\be_h,\eta_h|||_{B,h} \leq C (|||\bu-\pi_h \bu, 0|||_{B,h} + \mu^{\frac12} h |\bu|_{H^2(\Omega)} ).
\]
Applying the
approximation properties of the projection $\pi_h$ from Lemma \ref{lem:approx1}
now proves the claim.
\end{proof}
\subsection{Superconvergence of the primal variable in the Darcy
  limit}
Here we will prove that in the Darcy limit, the pressure variable
converges to $\pi_0 p$ with the rate $O(h^2)$ on convex domains. To
fix the ideas we consider \eqref{eq:brinkman} with $\sigma=1$ and 
$\mu=0$ and the boundary condition \eqref{eq:no_inflow}. We let
$(\bu_h,p_h)$ denote the solution of \eqref{eq:FEM_Brinkman_nitsche}. Not
that in this case we solve the problem $-\Delta p = g$ with $\nabla p \cdot \bn\vert_{\partial \Omega} =
0$. The following superconvergence result shows that we can use postprocessing to obtain a piecewise affine approximation of $p$ that has
optimal convergence in $H^1$ and $L^2$ norms.
\begin{prop}
Let $\Omega$ be convex. 
The following bound holds
\[
\|\pi_0 p - p_h\|_\Omega \leq C  ( h^2 \|g\|_\Omega + h\|g - \pi_0 g\|_\Omega).
\]
\end{prop}
\begin{proof}
Let $\varphi$ be the solution of the problem
\begin{equation}\label{eq:dual}
\begin{aligned} 
-\Delta \varphi &= \pi_0 p - p_h \\
\nabla \varphi \cdot \bn & = 0.
\end{aligned} 
\end{equation}
By the convexity assumption on $\Omega$ there holds by elliptic regularity
\begin{equation}\label{eq:reg}
\|p\|_{H^2(\Omega)} \leq C \|g\|_\Omega \mbox{ and } \|\varphi\|_{H^2(\Omega)} \leq C \|\pi_0 p - p_h\|_\Omega.
\end{equation}
By the definition of \eqref{eq:dual} we have
\[
\|\pi_0 p - p_h\|_\Omega^2 = (\pi_0 p - p_h, \Delta \varphi)_\Omega = (p
- p_h, \nabla \cdot \pi_h \nabla \varphi)_\Omega.
\]
By the definition of \eqref{eq:FEM_Brinkman_nitsche} there holds
\[
(p
- p_h, \nabla \cdot \pi_h \nabla \varphi)_\Omega = (\bu - \bu_h, \pi_h
\nabla \varphi)_\Omega .
\]
Now we add and subtract $\nabla \varphi$ in the right hand side to
obtain
\[
 (\bu - \bu_h, \pi_h
\nabla \varphi - \nabla \varphi )_\Omega +(\bu - \bu_h, 
\nabla \varphi)_\Omega= I + II.
\]
Using the Cauchy-Schwarz inequality and the interpolation properties of $\pi_h$ we see that
\[
I \leq \|\bu - \bu_h\|_\Omega C h |\varphi|_{H^2(\Omega)}.
\]
For term $II$ we integrate by parts and use once again the definition
of \eqref{eq:FEM_Brinkman_nitsche}.
\begin{equation*}
II \leq (\nabla \cdot (\bu - \bu_h), \varphi - \pi_0 \varphi)_\Omega
 = (g - \pi_0 g, \varphi - \pi_0 \varphi)_\Omega
\leq \|g - \pi_0 g\|_\Omega C h |\varphi|_{H^1(\Omega)}.
\end{equation*}
Collecting the above inequalities we see that using the error estimate
\eqref{prop:apriori_brinkman_nit} and \eqref{eq:reg} there holds
\begin{multline*}
\|\pi_0 p - p_h\|_\Omega^2 \leq C h (|||\bu - \bu_h,0|||_B + \|g -
\pi_0 g\|_\Omega) \|\varphi\|_{H^2(\Omega)}\\
\leq C ( h^2 \|g\|_\Omega + h\|g -
\pi_0 g\|_\Omega) \|\pi_0 p - p_h\|_\Omega.
\end{multline*}
This concludes the proof.
\end{proof}
\subsection{Further remarks on using Nitsche's method for the imposition
  of slip conditions}
We will here discuss the imposition of the normal component of the
velocity using Nitsche's method in the context of Brinkman's
problems with slip boundary conditions. This is useful in cases where
the domain is not polyhedral. For simplicity we consider pure slip
boundary conditions
\begin{equation}\label{eq:slip1}
\bu \cdot \bn = 0\mbox{ and } \bft \,\sigma(\bu,p) \bn = 0 \mbox{ on } \partial \Omega,
\end{equation}
where $\sigma(\bu,p) := \mu \nabla \bu - p \mathbb{I}$. This problem
is well-posed in the space $V^0_{\bn}$.

This time the Nitsche formulation 
takes the form: Find $(\bu_h, p_h) \in V_h \times Q_h$ such that
\begin{equation}\label{eq:FEM_elasticity_nit}
A_{B,h}(\bu_h,\bv_h) = l_B(\bv_h,q_h) 
\end{equation}
with
\begin{equation}\label{def:A_B}
A_{B,h}(\bu_h,\bv_h):=A_B(\bu_h,\bv_h) - c((\bu_h,p_h),\bv_h)-c((\bv_h,0),\bu_h)+s(\bu_h,\bv_h)
\end{equation}
where
\[
c((\bu_hp_h),\bv_h):= (\bn \cdot  \bsigma(\bu_h,p_h)  \bn, \bv_h \cdot \bn)_{\partial \Omega}
\]
and
\[
s(\bu_h,\bv_h) :=(\gamma/h (\mu + \sigma) ~\bu_h \cdot \bn, ~ \bv_h \cdot \bn)_{\partial
  \Omega}.
\]
Observe that to avoid perturbing the mass conservation the pressure
test function is absent in the second  $c$-form of the definition
\eqref{def:A_B}. This destroys the anti-symmetry of the pressure
velocity coupling in the boundary terms. One may however prove that
inf-sup stability of the norm defined in \eqref{def:3B} still holds for $h$ small enough.
Proceeding as in the proof of Proposition 6.3, using, inf-sup
stability, Galerkin orthogonality and continuity, one may then prove the
following a
priori error estimate.

\begin{prop}\label{prop:apriori_brinkman_nitb}
Under the hypothesis of Proposition \ref{prop:infsup_nit},
let $(\bu,p) \in V \times Q$ be the solution to \eqref{eq:brinkman},
with either $\mu>0$ and $\sigma \ge 0$ or  $\mu \ge 0$ and $\sigma > 0$
and $(\bu_h,p_h) \in V_h \times Q_h$ the solution to \eqref{eq:FEM_brinkman}. Then there
holds
\[
|||\bu-\bu_h,p-p_h|||_{B,h} \leq C h (\mu^{\frac12}|\bu|_{H^2(\Omega)}+ 
\sigma^{\frac12} |\bu|_{H^1(\Omega)}+
\gamma^{-\frac12} (\sigma+\mu)^{-\frac12} |p|_{H^1(\Omega)}).
\]
\end{prop}

\begin{rem}
We observe that the error estimate of Proposition
\ref{prop:apriori_brinkman_nitb} is less robust than that of
Proposition \ref{prop:apriori_brink}, since in the former the
pressure appears in the right hand side. This is due to the appearance
of a term $(p - \pi_0 p, \bw_h \cdot \bn)_{\partial \Omega}$ after
application of Galerkin orthogonality. This term can
not be eliminated through the choice of $\pi_0$, since this
approximation already has been fixed by imposing orthogonality on the
bulk of each element.
Note however that under our
assumptions either $\mu$ or $\sigma$ must be strictly positive
and therefore the constant can not degenerate. It can also be made as
small as desired by choosing the penalty parameter $\gamma$ large. Moreover, in the Darcy limit
$\sigma^{\frac12} \bu \sim \sigma^{-\frac12} \nabla p$ and therefore
the term $\sigma^{\frac12} |\bu|_{H^1(\Omega)} \sim \sigma^{-\frac12}
|\nabla p|_{H^1(\Omega)}  >> \sigma^{-\frac12}
|p|_{H^1(\Omega)} $ and it follows that the pressure contribution is
the lower order term. The limit where both $\mu$ and
$\sigma$ go to zero simultaneously is not physically relevant.
\end{rem}

\section{Numerical examples}\label{sec:numerics}

In this Section we provide some details on the practical implementation of the approximation and give numerical examples
of near incompressible elasticity, Stokes flow, Darcy flow, and
coupled Darcy-Stokes flow. For simplicity, we
  consistently use strong imposition of boundary conditions in the
  examples.

\subsection{Elasticity}\label{sec:numerics_elasticity}

We consider the well known Cook's membrane, which is a quadrilateral with corners at (0,0), (48,44), (48,60), and (0,44), in a condition of plane strain.
The quadrilateral is fixed, $\bu=(0,0)$, at $x=0$, has zero traction, $\bsigma(\bu)\cdot\bn=(0,0)$, on the upper and lower boundary, and $\bsigma(\bu)\cdot\bn = (0,1)$ (a vertical shearing load) at $x=48$. This particular choice of boundary traction and a Young's modulus of $E=200$, is
taken from \cite{ChCeCo15}. Cook's membrane is highly susceptible to locking in the incompressible limit for low order elements as we illustrate in Fig. \ref{fig:locking}, where we compare the present method to a standard piecewise linear approximation on the type I triangles in the same mesh. The standard method locks as $\nu\rightarrow 0.5$, whereas the present method is unaffected. The results compare well with those of \cite{ChCeCo15}. In Fig. \ref{fig:def} we show the mesh of macro triangles and the computed deformation obtained the present method.

\subsection{Stokes flow}\label{sec:numerics_stokes}

We consider a problem on the unit square $(0,1)\times (0,1)$ with exact solution 
\[
\bu =(20 x y^3, 5x^4-5y^4), \quad p=60 y x^2-20y^3-5
\]
 with $\ff =(0,0)$ and Dirichlet boundary conditions
given by the exact solution. Zero mean pressure is enforced by a Lagrange multiplier.

In Fig. \ref{fig:convstokes} we show the convergence obtained with our method. The meshsize is defined as $1/\sqrt{\text{NNO}}$, where $\text{NNO}$ is the number of nodes on the grid of macro triangles. The dashed lines have inclination 1:1 and 1:2. The discrete solution on one of the meshes in the sequence is shown in Fig. \ref{fig:solstokes}.

\subsection{Darcy flow}\label{sec:numerics_darcy}

We consider a problem from \cite{MaTaWi02} on the unit square $(0,1)\times (0,1)$ with exact solution 
\[
\bu =(-\pi\sin^2{(\pi x)}\sin{(2\pi y)},\pi\sin{(2\pi x)}\sin^2{(\pi y)}), \quad p=\sin{(\pi x)}-2/\pi 
\]
given by 
 \[
 \ff =\left(\pi(\cos{(\pi x)} - \sin^2{(\pi x)}\sin{(2\pi y)})), \pi\sin{(2\pi x}\sin^2{(\pi y)}\right),
 \] 
and Dirichlet boundary conditions $\bu\cdot\bn = 0$ on the boundary.
Zero mean pressure is again enforced by a Lagrange multiplier.

In Fig. \ref{fig:convdarcy} we show the convergence obtained with our method. The meshsize is defined as in the previous example, as are the dashed lines. 
The discrete solution on one of the meshes in the sequence is shown in Fig. \ref{fig:soldarcy}.

\subsection{Coupled Stokes--Brinkman flow}\label{sec:numerics_Brinkman}

In this Section we show two examples of coupled Stokes--Brinkman flow. The domain is $(0,2)\times (0,2)$ in both cases. In the first example we show normal coupling. The boundary conditions are ${\boldsymbol u} ={\bf 0}$ at $x= 0$ and $x=2$. We let $\mu = 1$ and $\sigma = 0$ for $y \leq 1$.
At $y > 1$ we choose $\sigma  = 1$ and decrease $\mu$. We use a right--hand side ${\boldsymbol f} = (0,100)$. In figs. \ref{normal1} and \ref{normal2} we show the streamlines for successively decreasing $\mu \in \{1,10^{-2},10^{-3},10^{-6}\}$ on an $80\times 80$ nodes uniform mesh. The flow tends to uniform in the upper part and has to make a turn from a parabolic profile in the lower part at $y=1$. 

The second example concerns tangential coupling. The domain and right--hand side are the same, but the boundary conditions are ${\boldsymbol u}\cdot{\boldsymbol n}$ at $x=0$ and ${\boldsymbol u}={\bf 0}$ at $x=2$. Here we take $\mu =100$, $\sigma=0$ for $x > 1$ and $\sigma = 10^3$ with decreasing $\mu$ for $x\leq 1$.
In Figs. \ref{profile1}--\ref{profile2} we show the velocity profiles at $y=1$ for $\mu \in \{10,1,10^{-1},10^{-2}\}$ computed on an $80\times 80$ 
nodes uniform mesh. Note the oscillations occurring for decreasing $\mu$, related to the forced tangential continuity which cannot be upheld as
$\mu/\sigma \rightarrow 0$. The remedy for this effect (which will occur in the limit also for the normal coupling example) is to release tangential continuity or invoke an interface law relaxing tangential continuity using a physically motivated model \cite{GaHiGi96}, or using a variant of Nitsche's method as described above, cf. also \cite{BH07}.


\newpage
\begin{figure}[h]
\begin{center}
\includegraphics[scale=0.1]{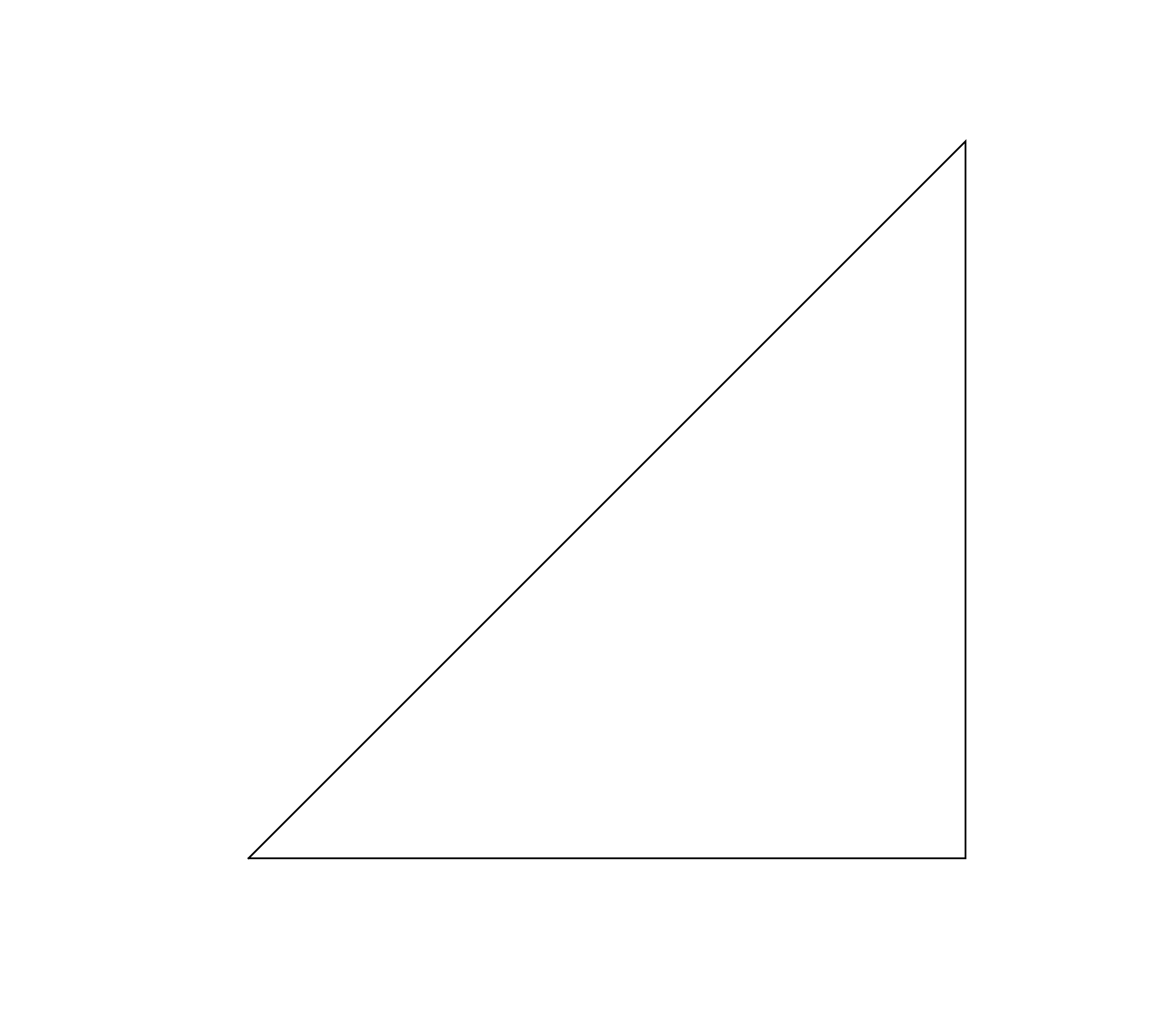}\includegraphics[scale=0.1]{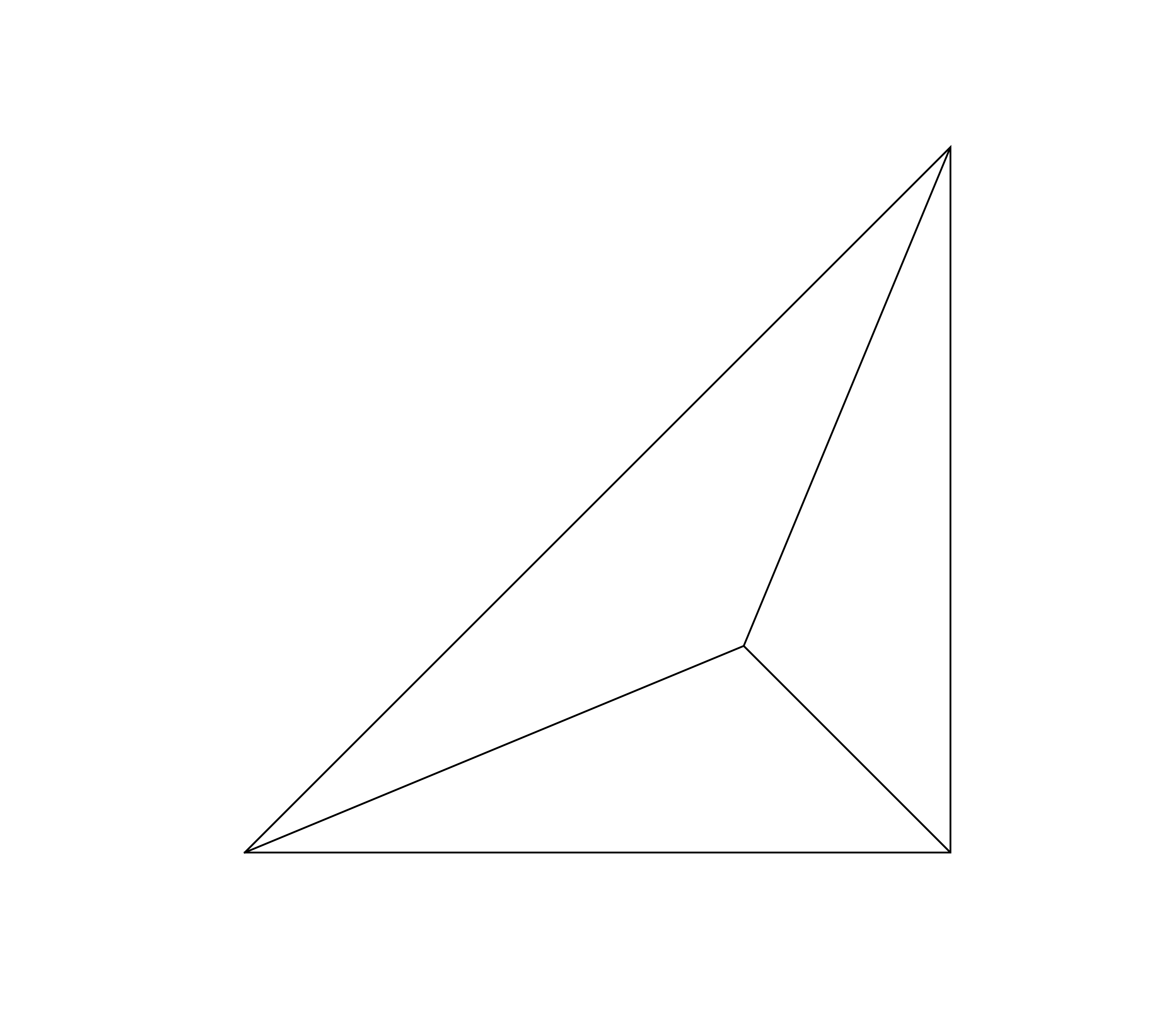}\includegraphics[scale=0.1]{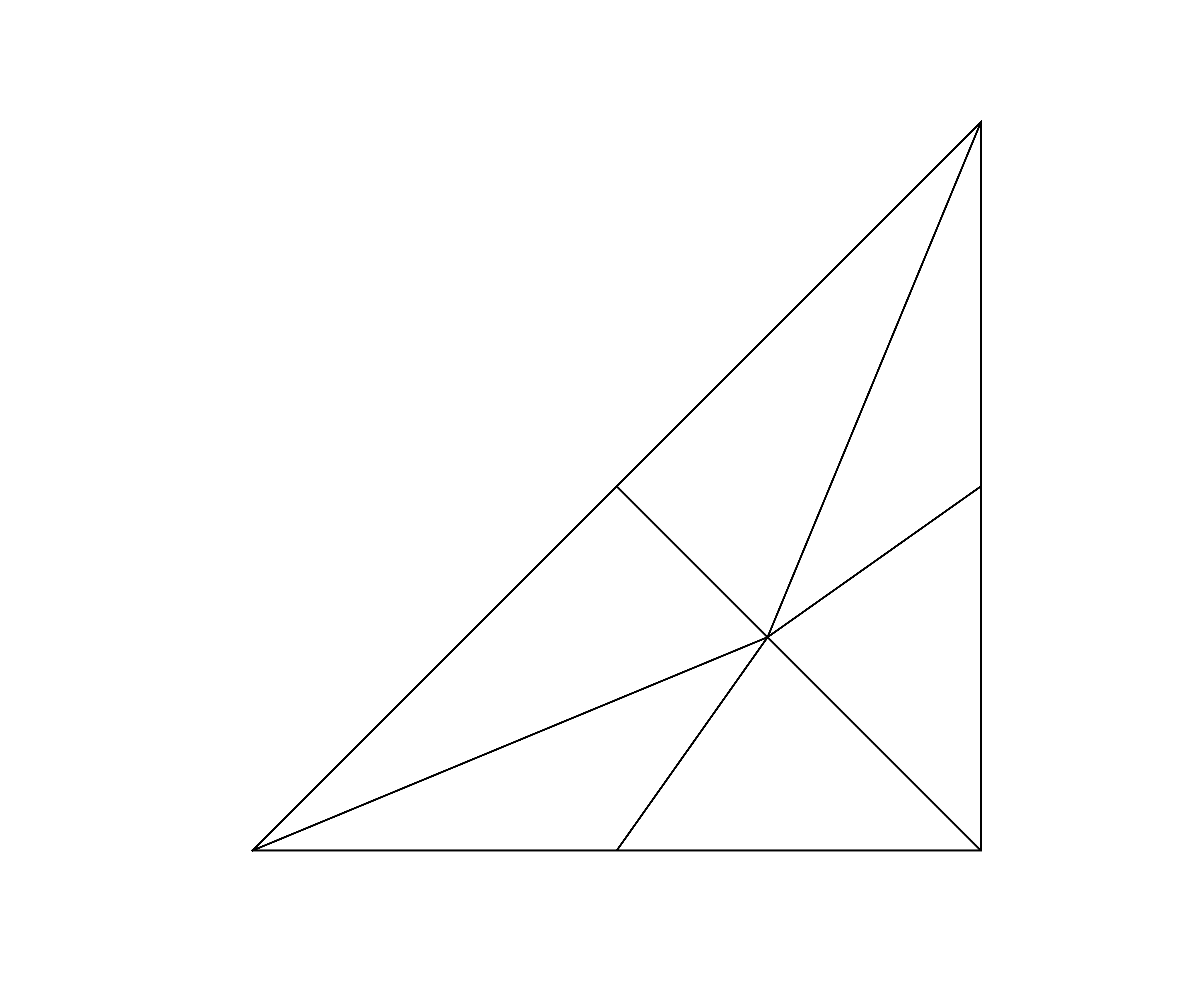}
\end{center}
\caption{The type I triangles (left) are once divided into type II triangles (middle) which are further divided to type III triangles (right).\label{fig:triangles}}
\end{figure}
\begin{figure}[h]
\begin{center}
\includegraphics[scale=0.3]{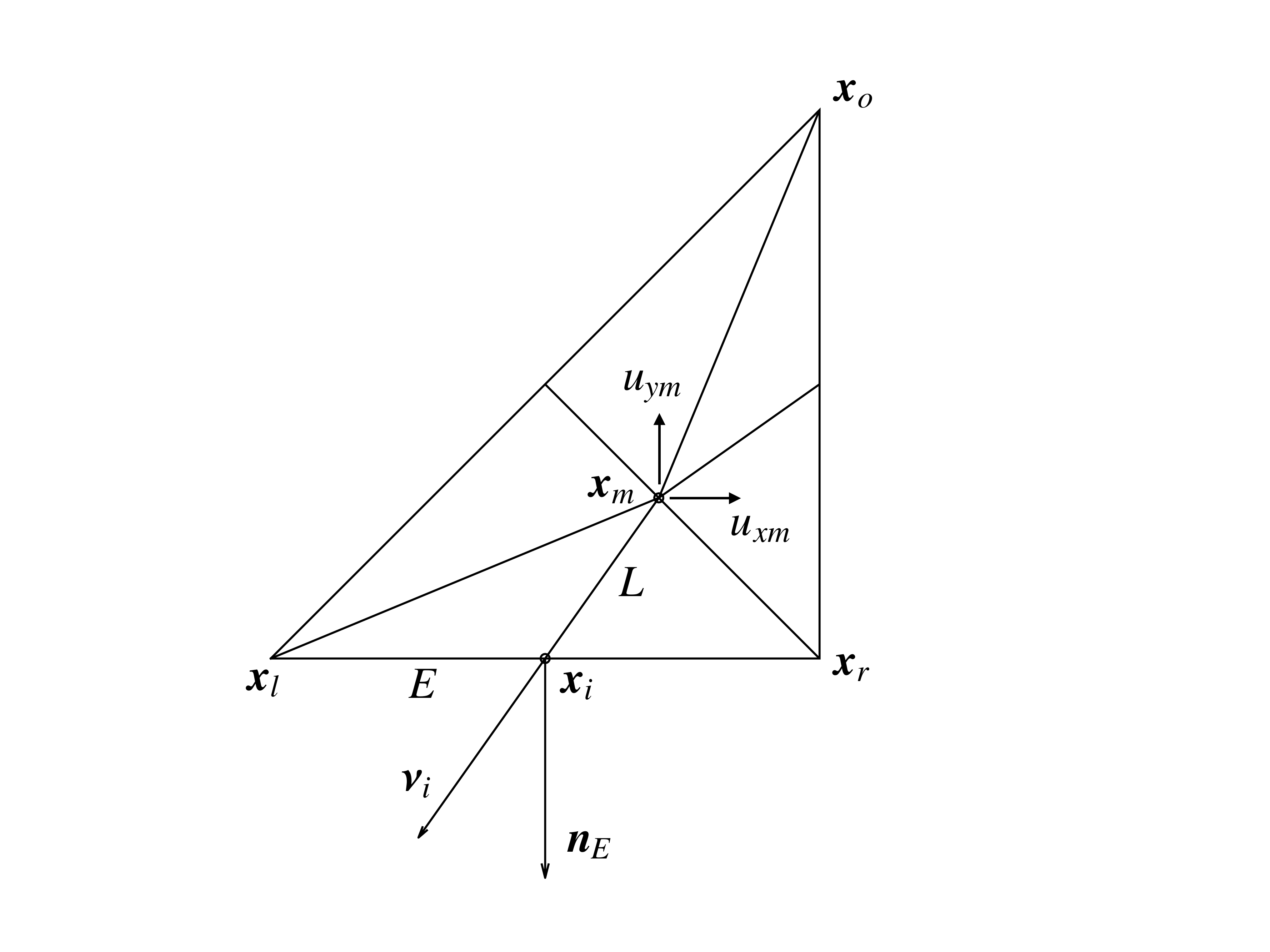}\end{center}
\caption{Quantities used to define the hierarchical bubble associated with edge $E$.\label{fig:bubbledef}}
\end{figure}
\begin{figure}[h]
\begin{center}
\includegraphics[scale=0.25]{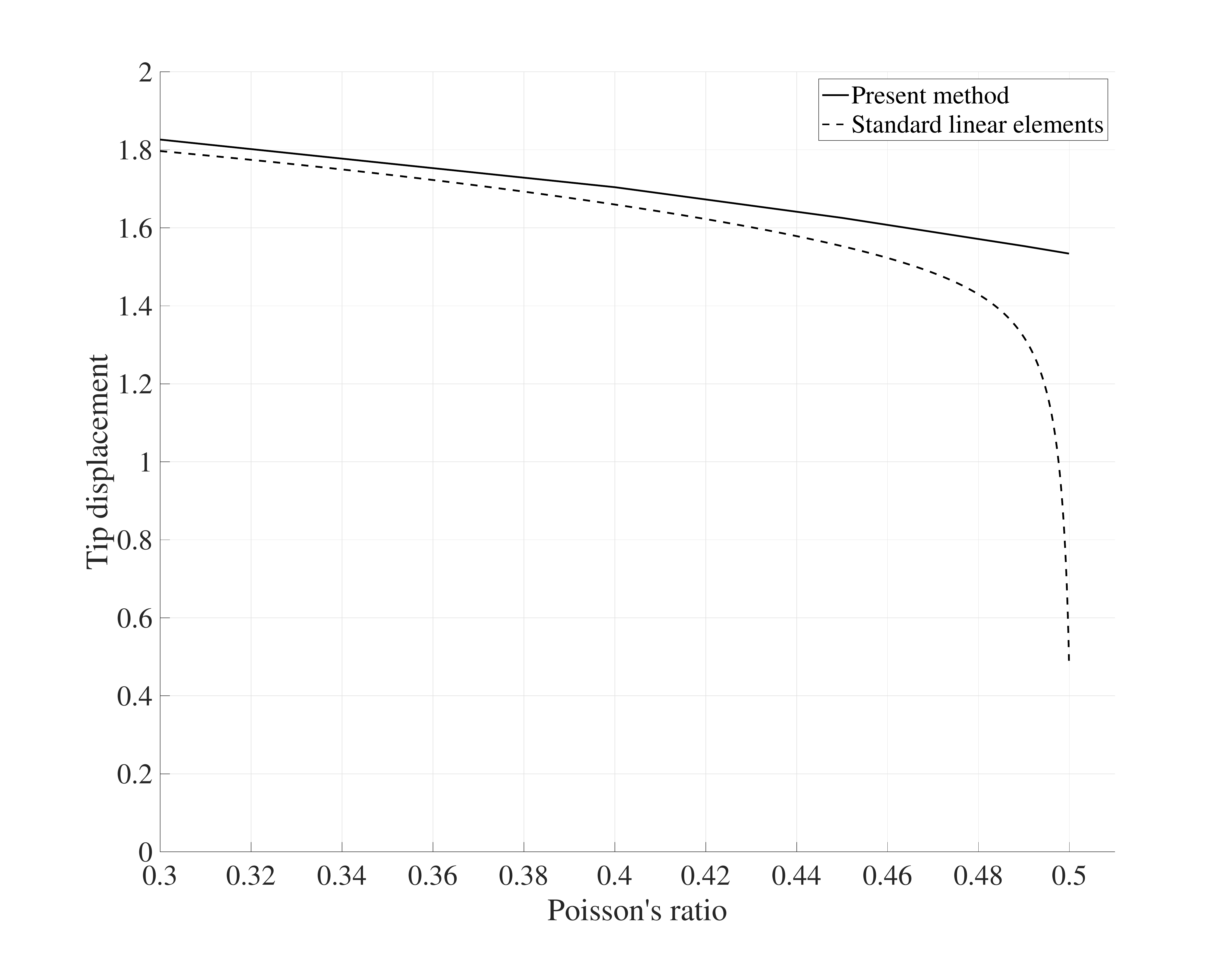}
\end{center}
\caption{Locking with standard linear elements and locking free solution with the present approximation.\label{fig:locking}}
\end{figure}

\begin{figure}[h]
\begin{center}
\includegraphics[scale=0.25]{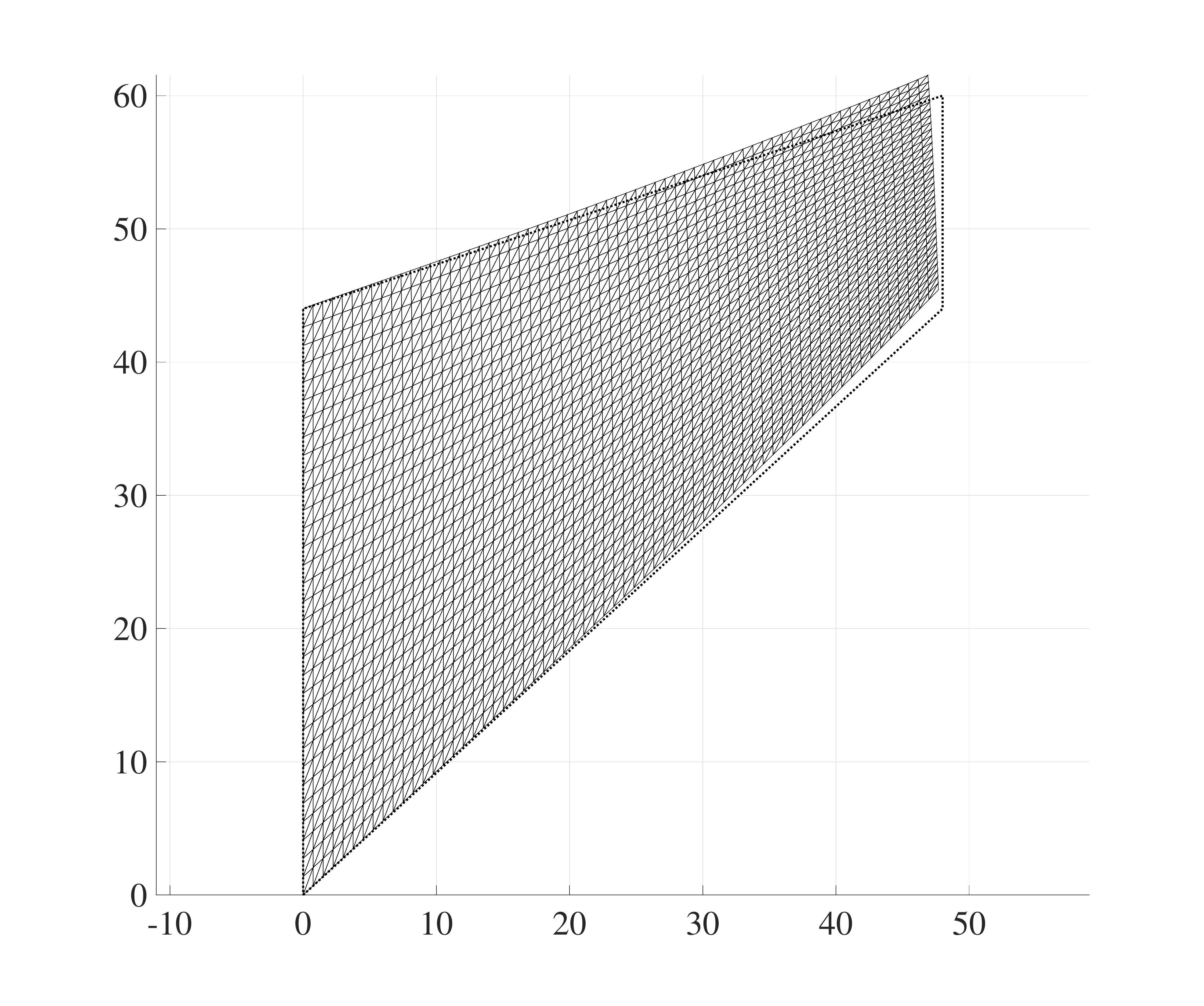}
\end{center}
\caption{Mesh and corresponding solution for $\nu=0.49999$.\label{fig:def}}
\end{figure}

\begin{figure}[h]
\begin{center}
\includegraphics[scale=0.2]{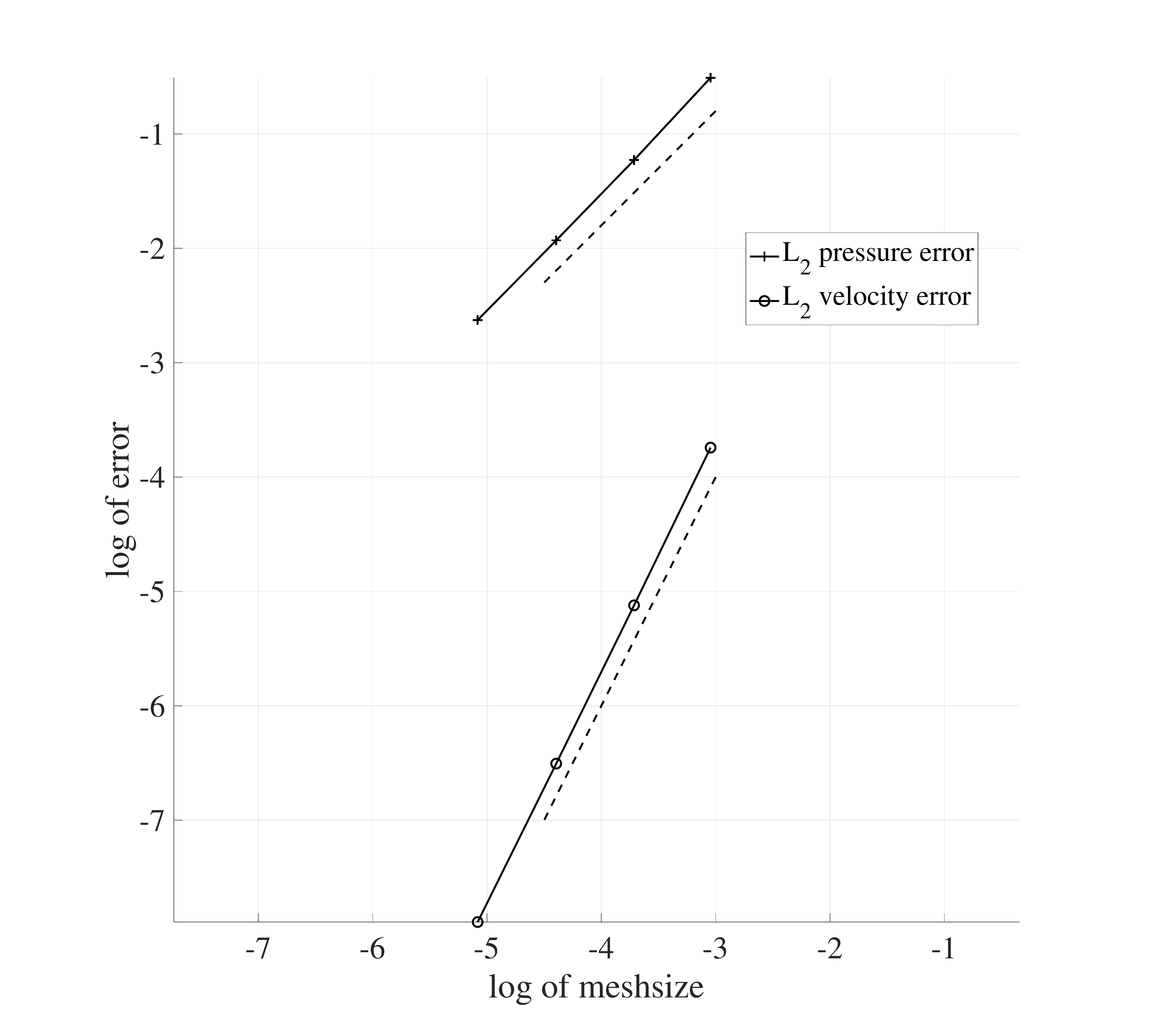}
\end{center}
\caption{Convergence for a Stokes problem on a sequence of meshes.\label{fig:convstokes}}
\end{figure}

\begin{figure}[h]
\begin{center}
\includegraphics[scale=0.2]{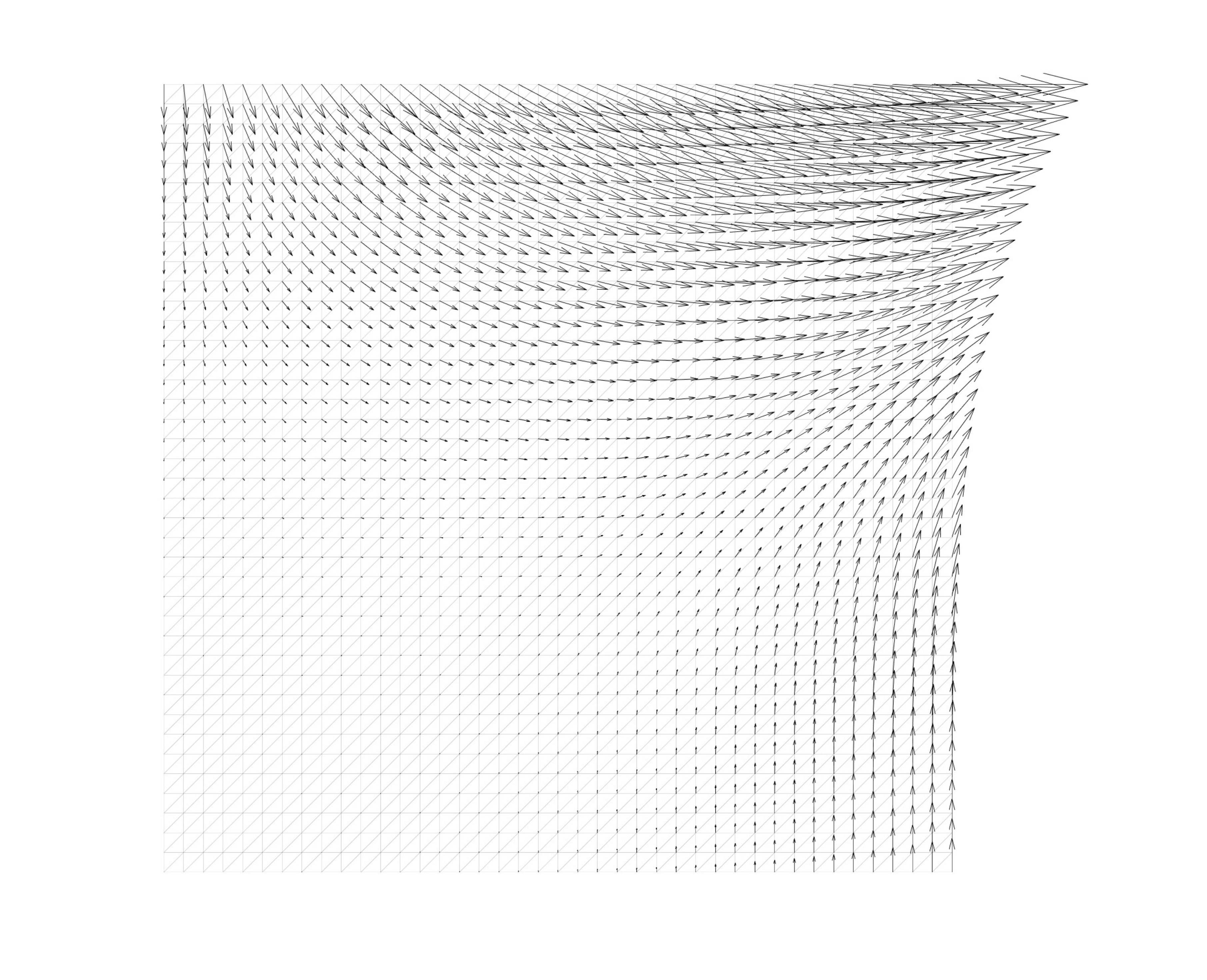}\includegraphics[scale=0.15]{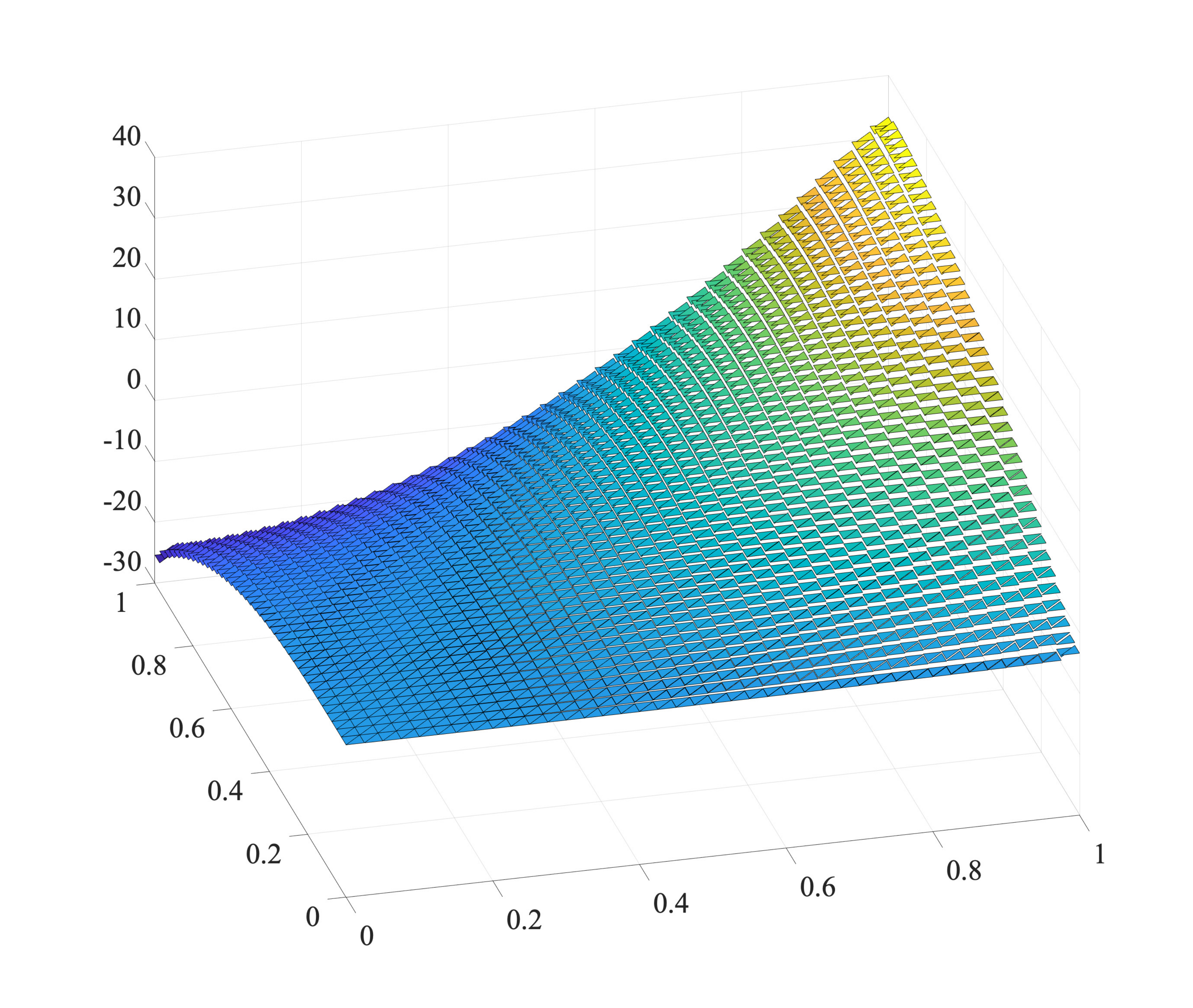}
\end{center}
\caption{Velocity and pressure solutions on a mesh in the sequence.\label{fig:solstokes}}
\end{figure}

\begin{figure}[h]
\begin{center}
\includegraphics[scale=0.2]{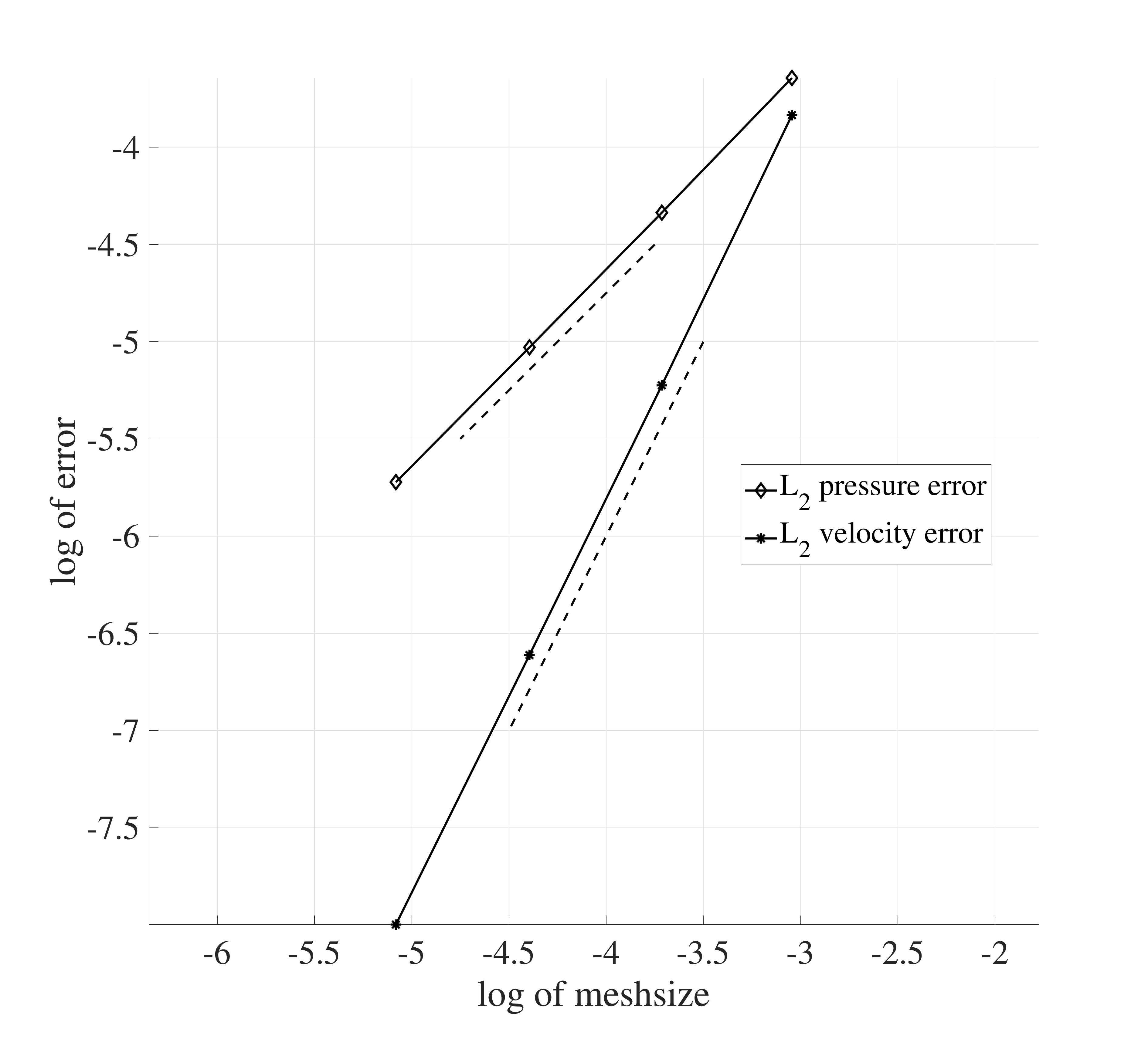}
\end{center}
\caption{Convergence for a Darcy problem on a sequence of meshes.\label{fig:convdarcy}}
\end{figure}
\begin{figure}[h]
\begin{center}
\includegraphics[scale=0.2]{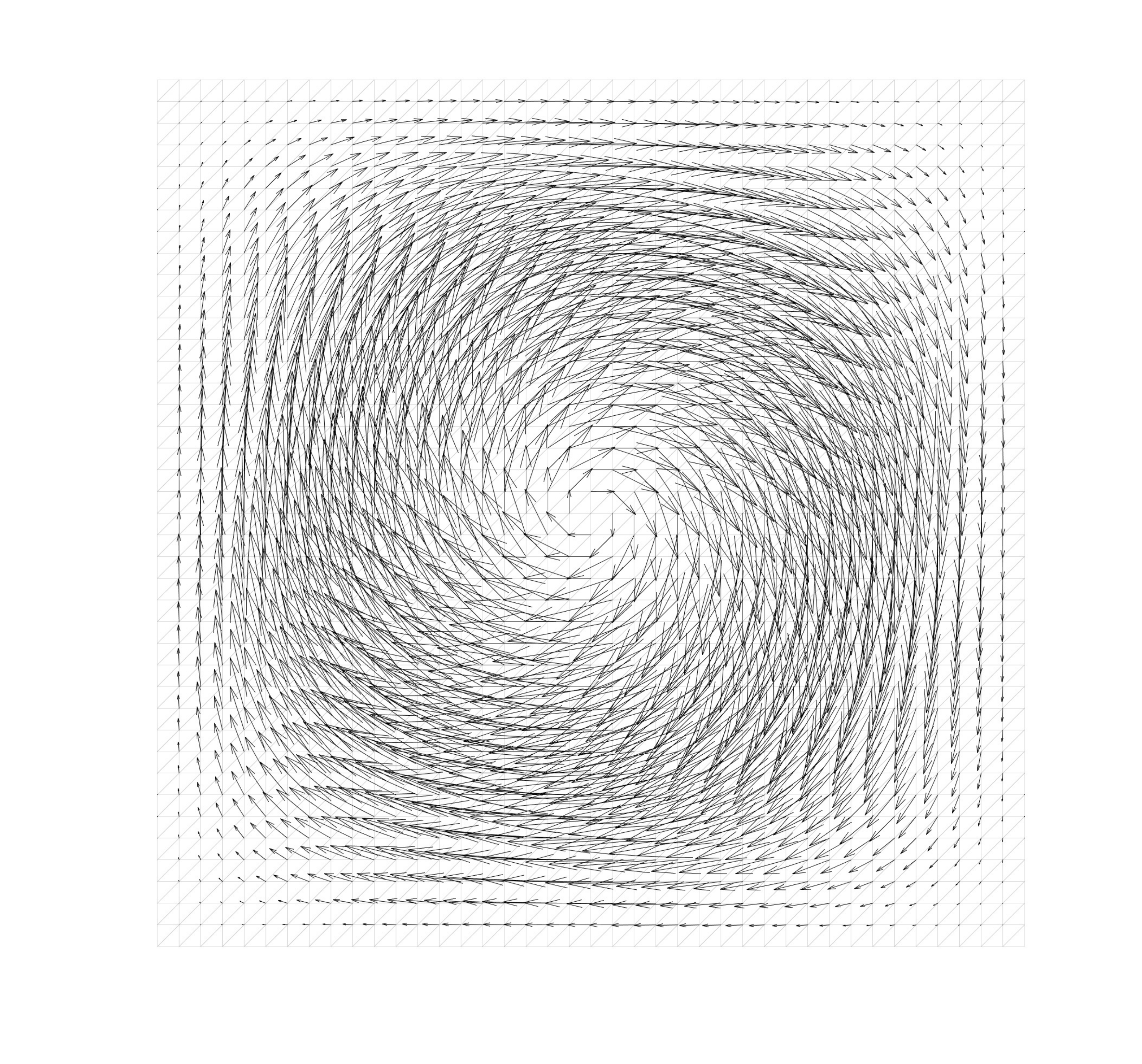}\includegraphics[scale=0.2]{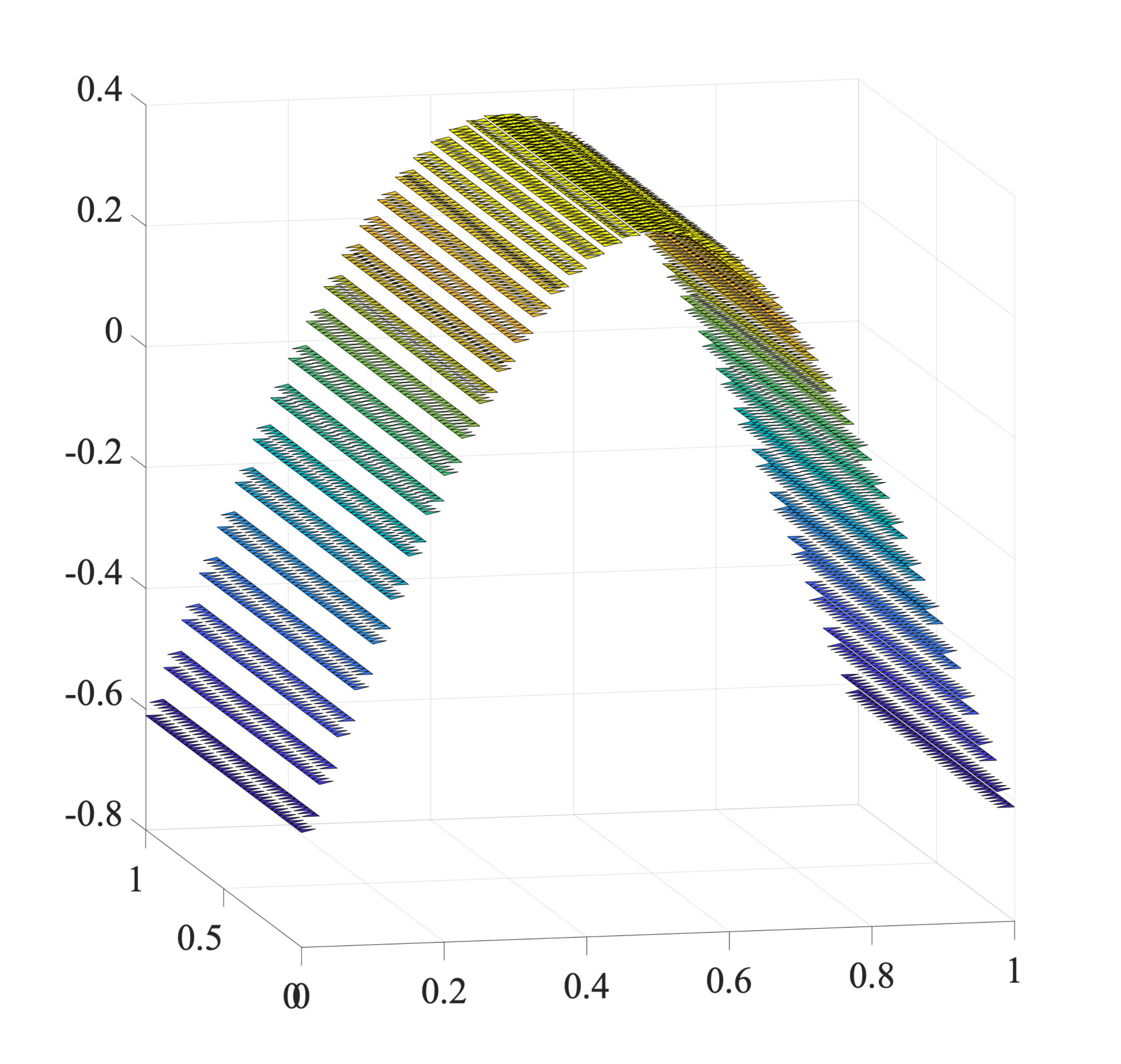}
\end{center}
\caption{Velocity and pressure solutions on a mesh in the sequence.\label{fig:soldarcy}}
\end{figure}
\begin{figure}[h]
\begin{center}
\includegraphics[scale=0.33]{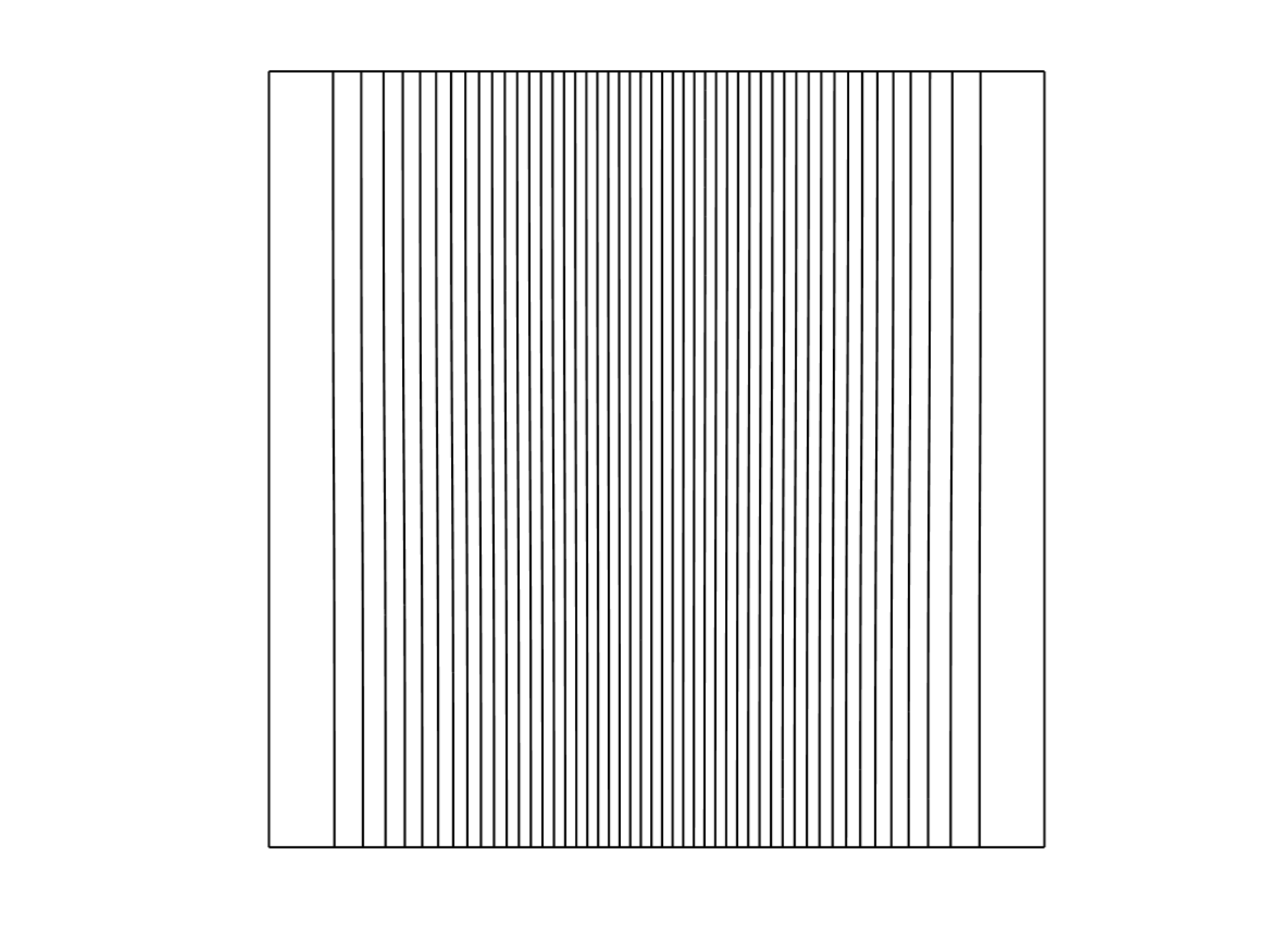}
\includegraphics[scale=0.33]{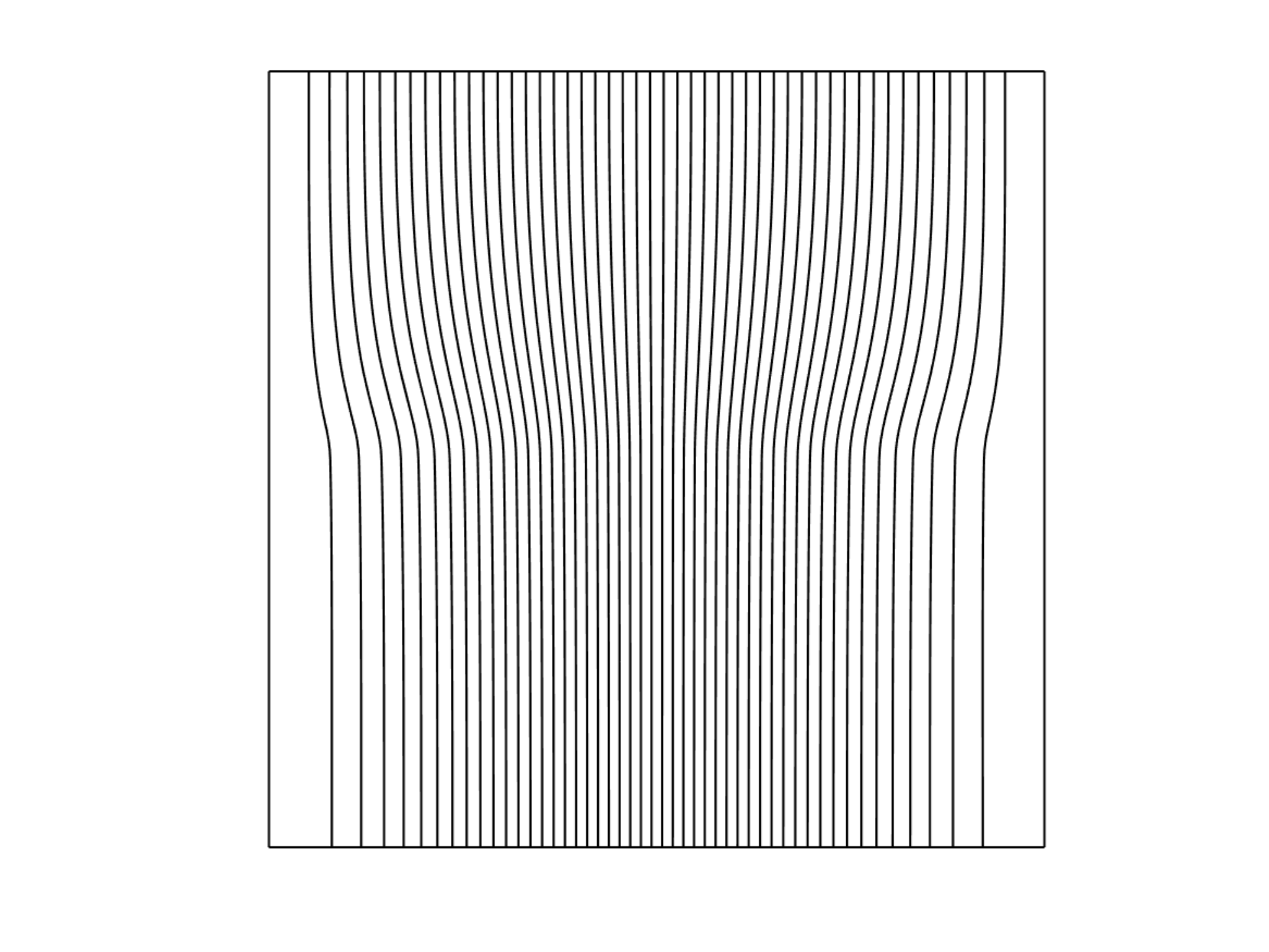}
\end{center}
\caption{Streamlines for $\mu = 1$ and $\mu = 10^{-2}$.\label{normal1}}
\end{figure}
\begin{figure}[h]
\begin{center}
\includegraphics[scale=0.33]{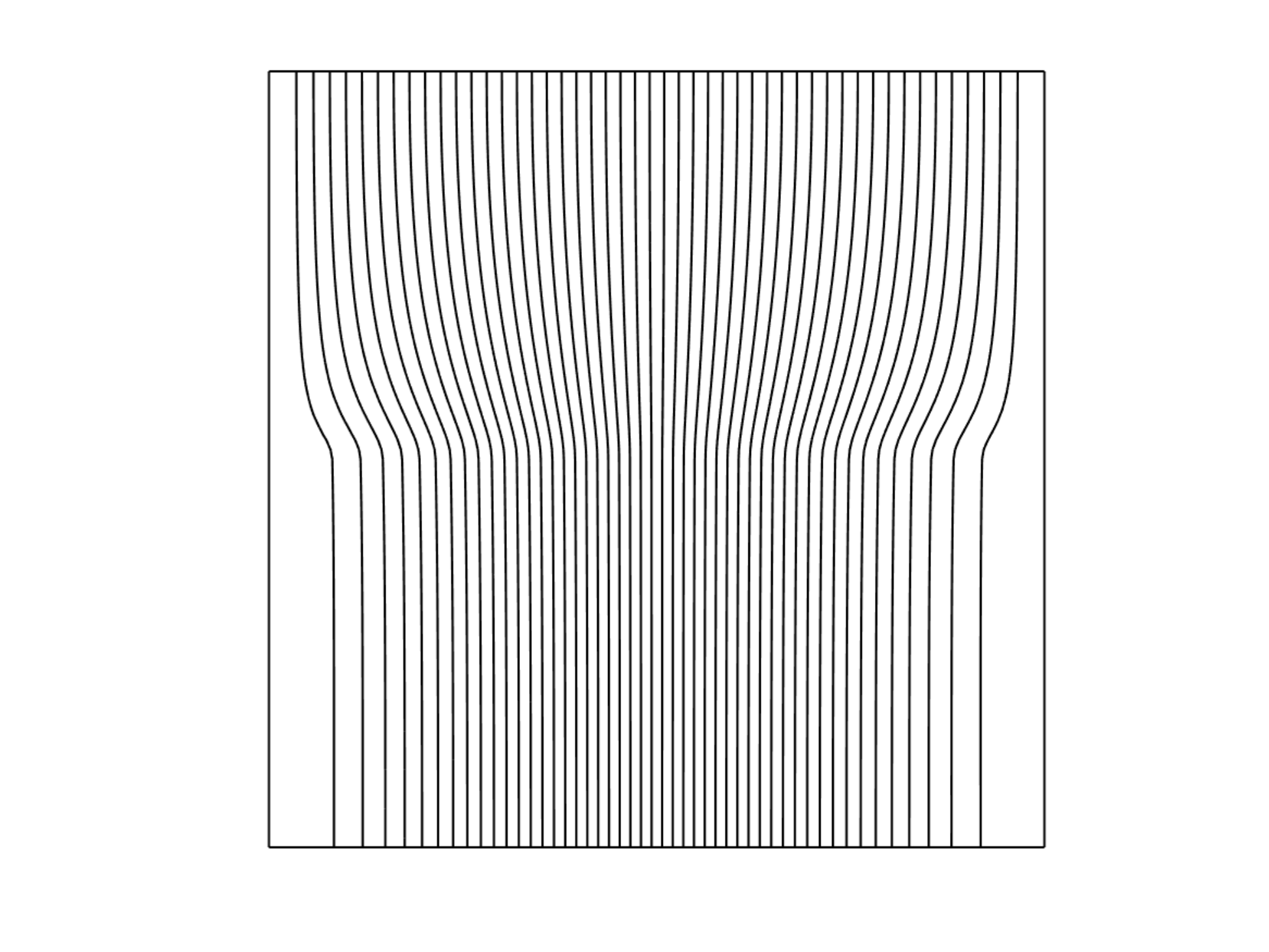}\includegraphics[scale=0.33]{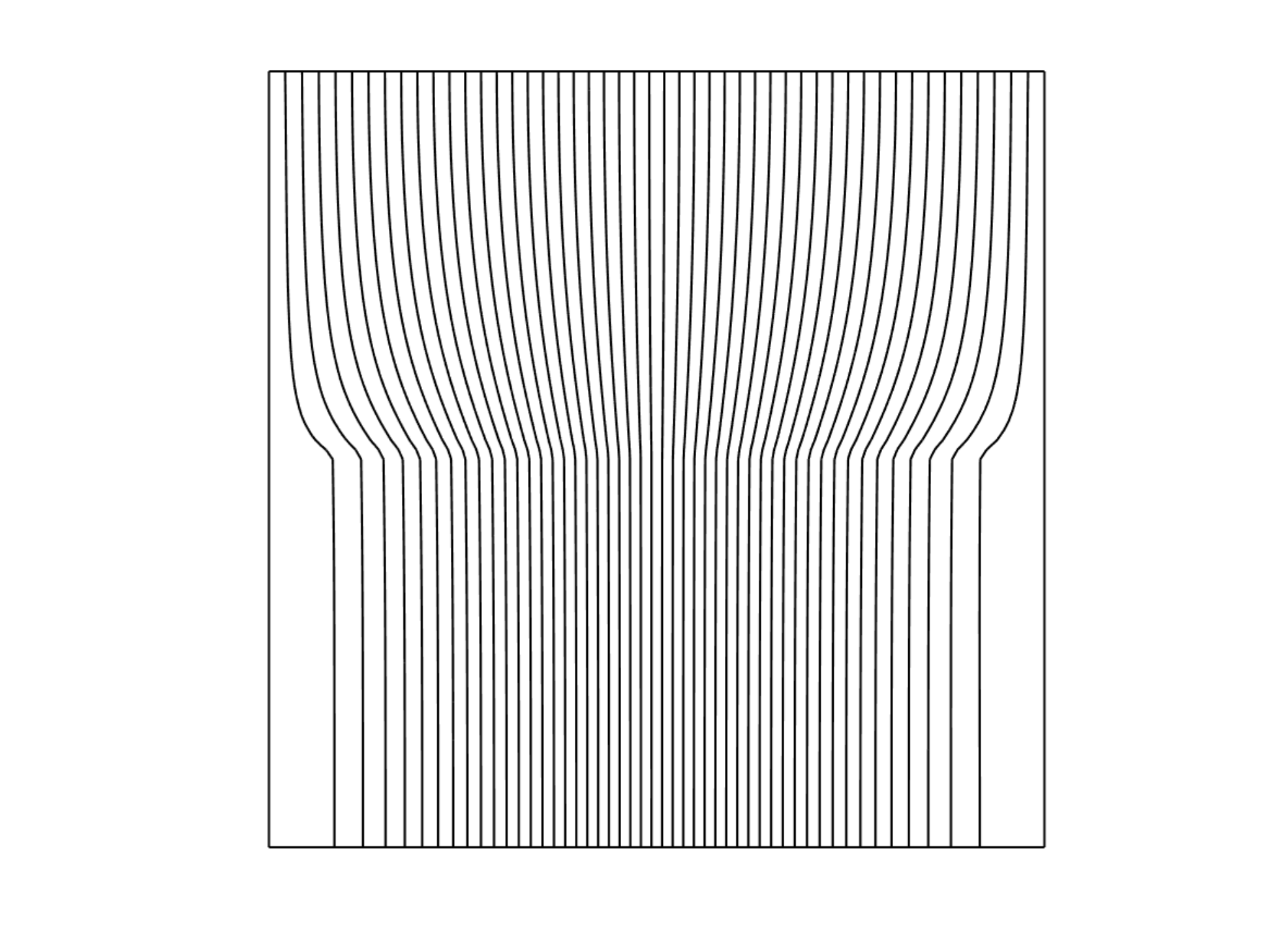}
\end{center}
\caption{Streamlines for $\mu = 10^{-3}$ and $\mu = 10^{-6}$.\label{normal2}}
\end{figure}
\begin{figure}[h]
\begin{center}
\includegraphics[scale=0.17]{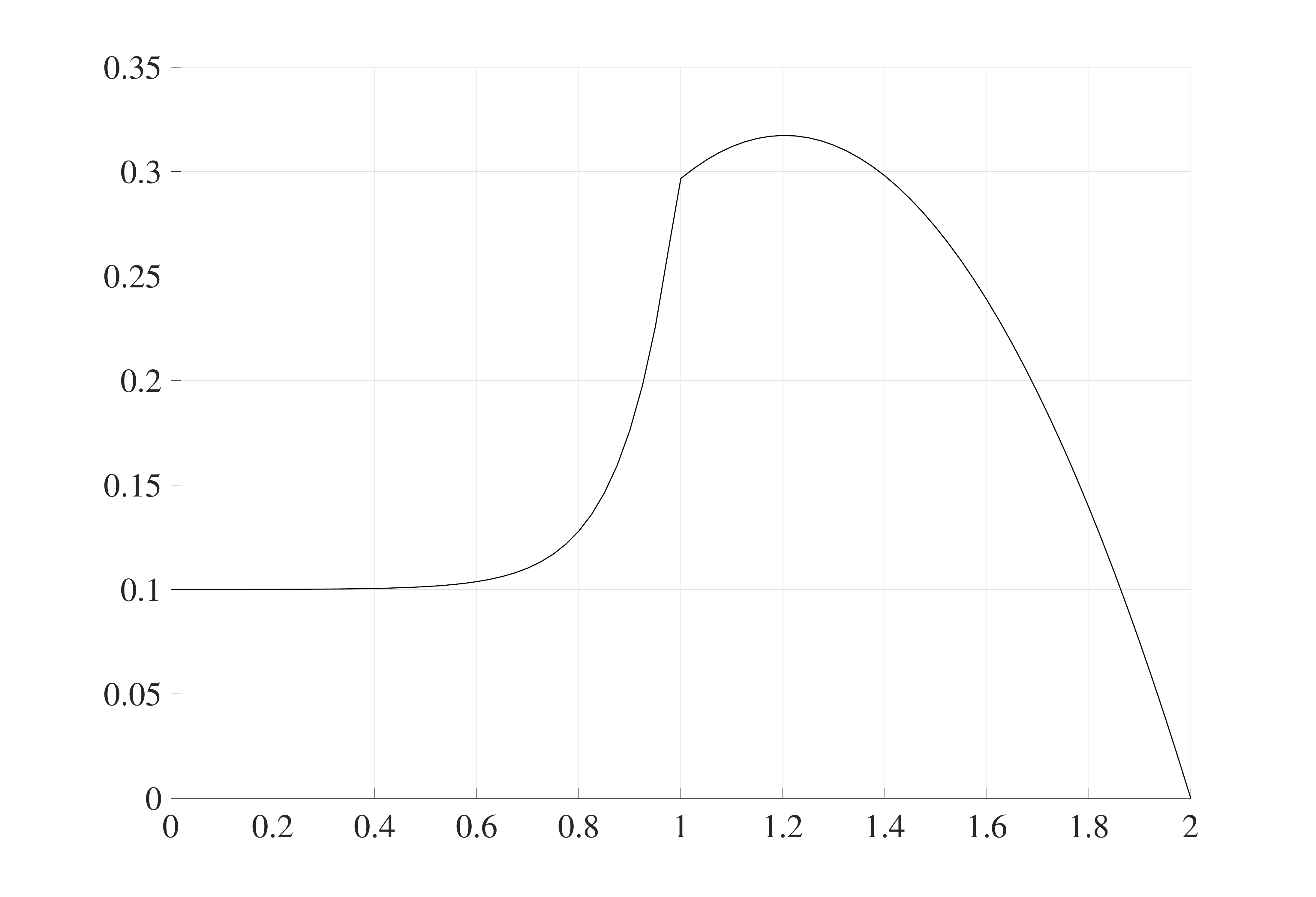}\includegraphics[scale=0.17]{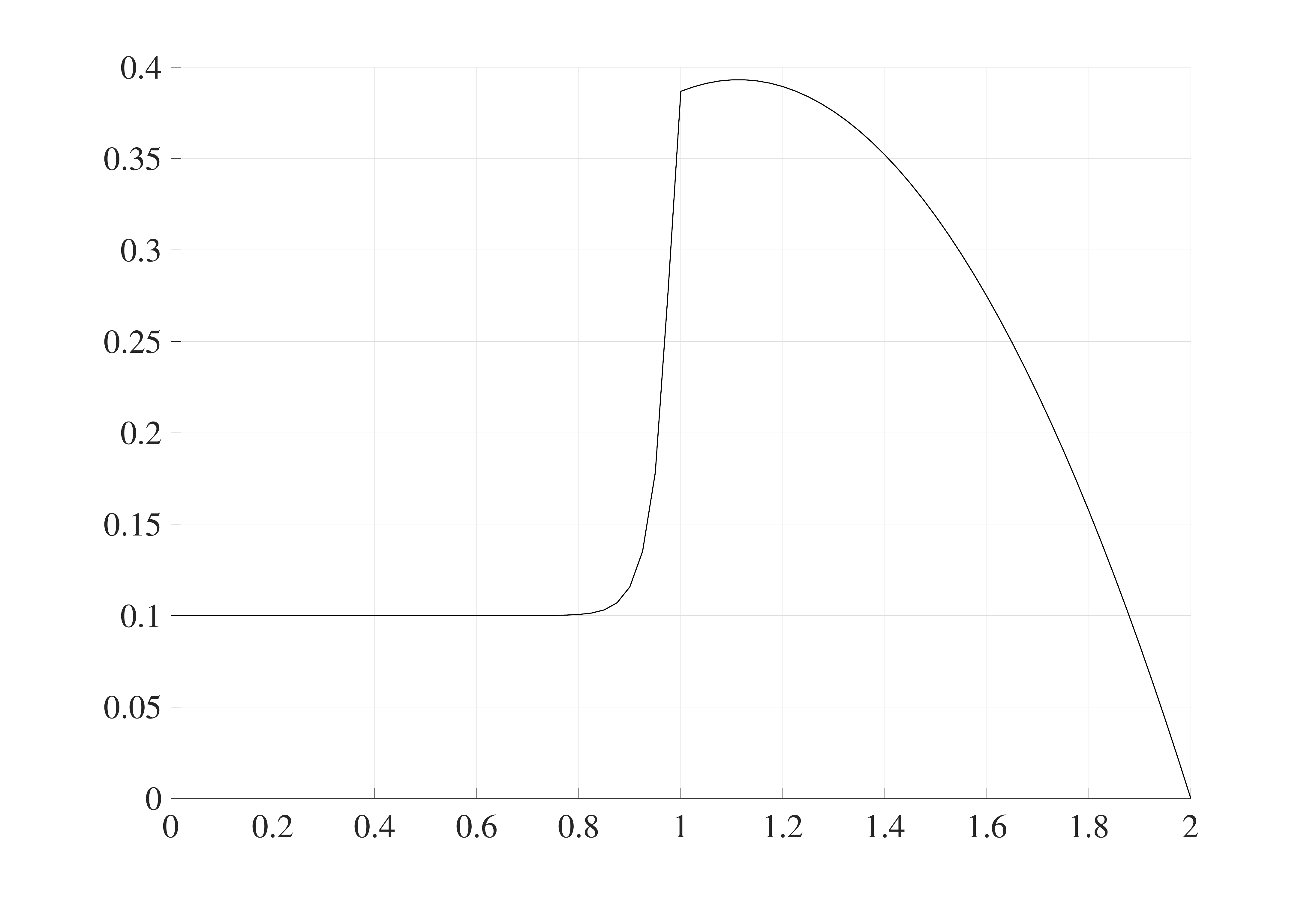}
\end{center}
\caption{Velocity profiles for $\mu = 10$ and $\mu = 1$.\label{profile1}}
\end{figure}
\begin{figure}[h]
\begin{center}
\includegraphics[scale=0.17]{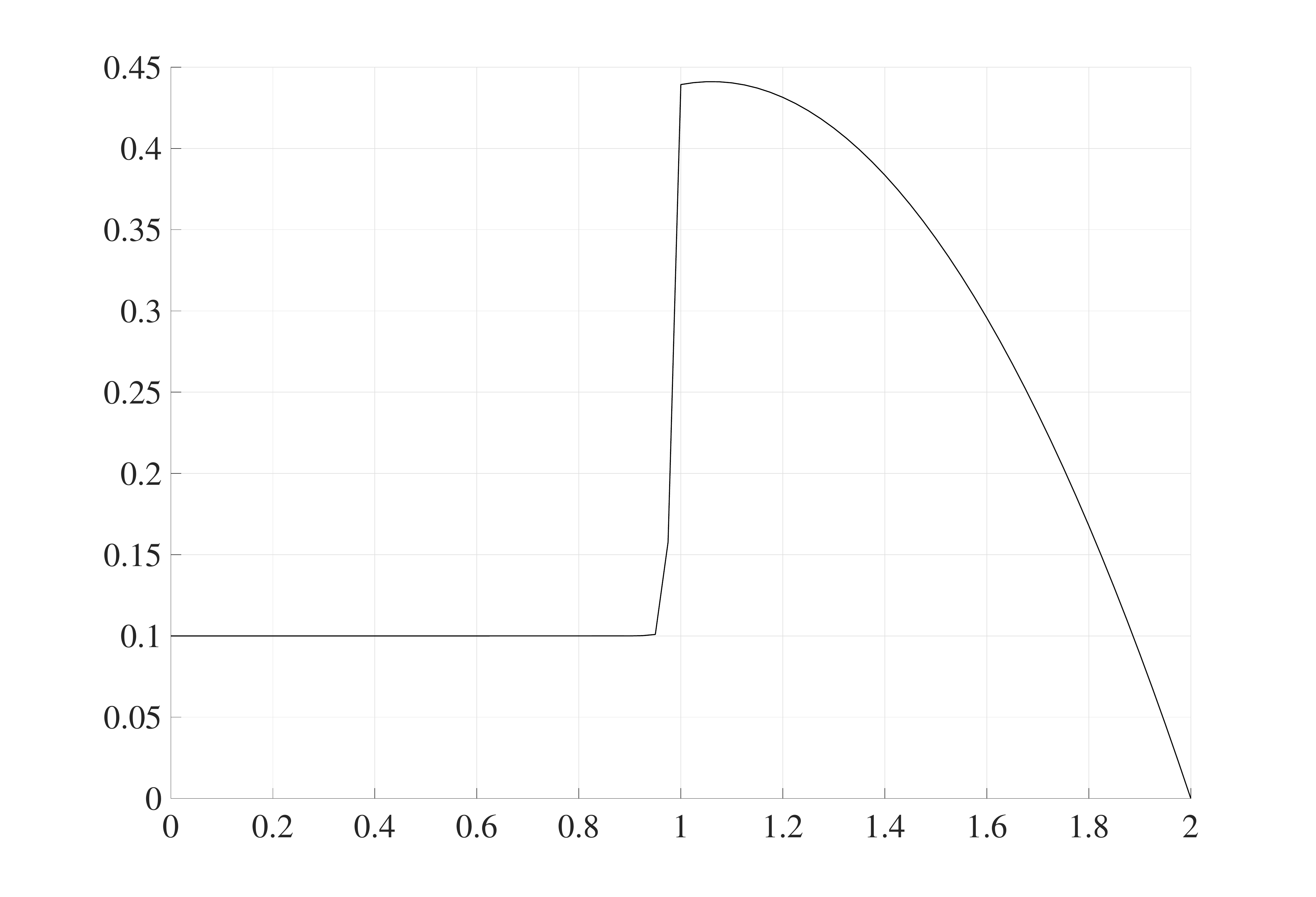}\includegraphics[scale=0.17]{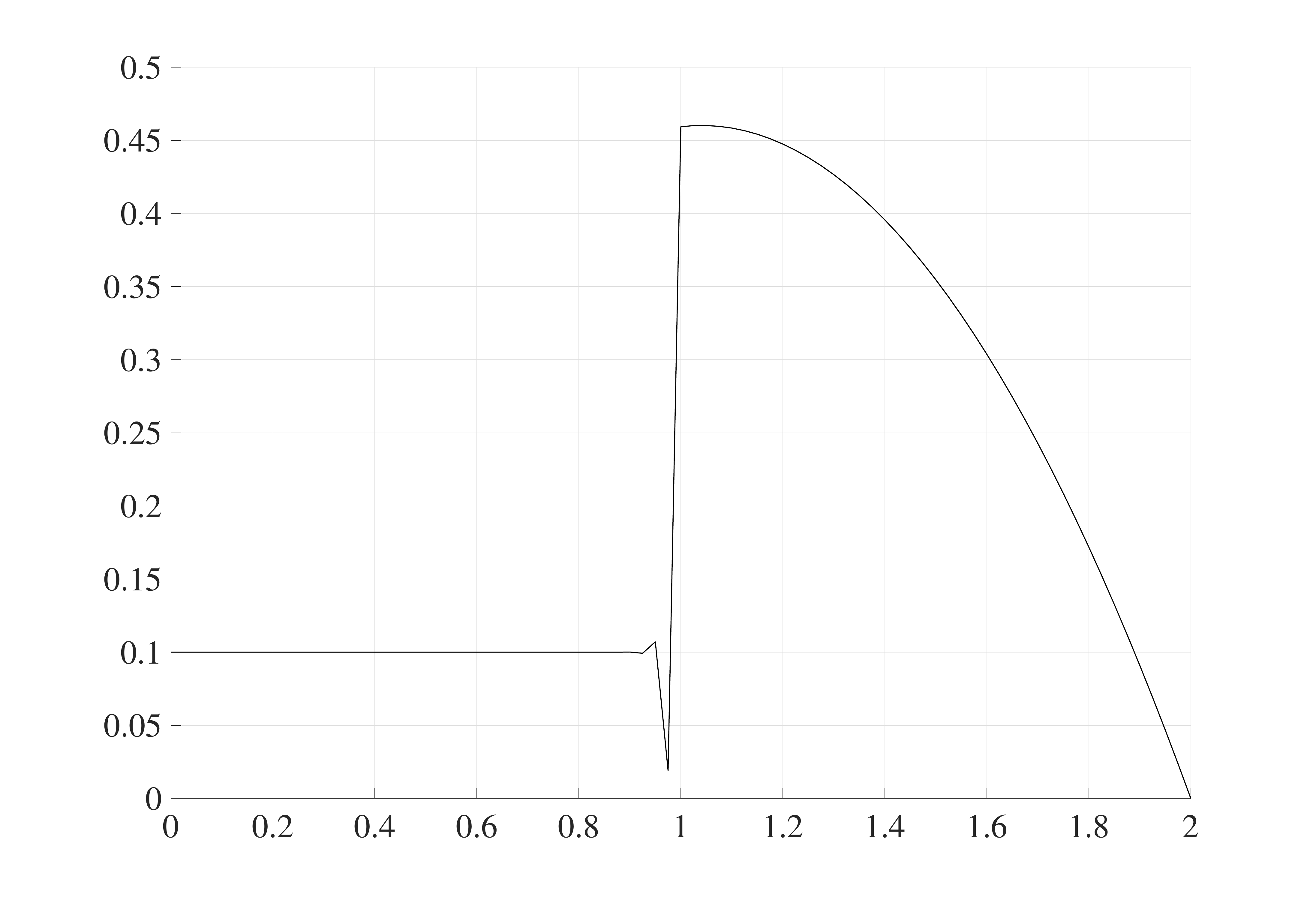}
\end{center}
\caption{Velocity profiles for $\mu = 10^{-1}$ and $\mu = 10^{-2}$.\label{profile2}}
\end{figure}

\end{document}